\documentclass[reqno]{amsart}
\usepackage{amssymb}
\usepackage{epsfig}  		
\usepackage{epic,eepic}    
\usepackage[colorlinks=true]{hyperref}
\hypersetup{urlcolor=blue, citecolor=blue}
\usepackage{mathrsfs}
\usepackage[latin1]{inputenc}
\usepackage[T1]{fontenc}
\usepackage{color,xspace}
\usepackage[usenames,dvipsnames]{pstricks}
\allowdisplaybreaks

\usepackage{comment}

\newtheorem{theorem}{Theorem}[section]

\newtheorem{corollary}[theorem]{Corollary}

\newtheorem{lemma}[theorem]{Lemma}

\newtheorem{proposition}[theorem]{Proposition}

\newcommand{\R}{\mathbb{R}}
\newcommand{\C}{\mathbb{C}}

\newcommand{\N}{\mathbb{N}}

\newcommand{\ds}{\mathrm{ds}}

\newcommand{\h}{\mathcal{H}}
\newcommand{\s}{\mathbb{S}}

\newcommand{\p}{\partial}
\newcommand{\Dd}{\nabla}
\newcommand{\norm}[1]{\|#1\|}
\newcommand{\abs}[1]{\left|#1\right|}
\newcommand{\set}[1]{\left\{#1\right\}}
\newcommand{\para}[1]{\left(#1\right)}

\newcommand{\seq}[1]{\left<#1\right>}

\newcommand{\data}{\Vert \mathcal{N}_{A_1,q_1}-\mathcal{N}_{A_2,q_2} \Vert}

\usepackage{graphicx}

\usepackage{mathrsfs}
\usepackage[latin1]{inputenc}
\usepackage[T1]{fontenc}
\usepackage{color,xspace}
\usepackage[usenames,dvipsnames]{pstricks}
\usepackage{times}
\allowdisplaybreaks
\newcommand{\sss}{\mathbb{S}^2}
\newcommand{\Rt}{\mathbb{R}^{3}}

\newcommand{\HAV}{\mathcal{H}_{A,q}}

\newcommand{\HAVB}{\mathcal{H}_{A_2,q_2}}

\newcommand{\rr}{\rho}

\begin{document}

\title[ Inverse time harmonic magnetic Schr\"odinger problem]{Stability estimate for an inverse problem for the time harmonic magnetic Schr\"odinger operator from the near and far field pattern }

\author[M.~Bellassoued]{Mourad Bellassoued}
\author[H.~Haddar]{Houssem Haddar}
\author[A.~Labidi]{Amal Labidi}

\address{M.~Bellassoued. Universit\'e de Tunis El Manar, Ecole Nationale d'ing\'enieurs de Tunis, ENIT-LAMSIN, B.P. 37, 1002 Tunis, Tunisia}
\email{mourad.bellassoued@enit.utm.tn }
\address{H.~Haddar, INRIA, ENSTA Paris Tech,  Institut Polytechnique de
Paris, Route de Saclay, 91128 Palaiseau, France}
\email{Houssem.Haddar@inria.fr}
\address{A.~Labidi. Universit\'e de Tunis El Manar, Ecole Nationale d'ing\'enieurs de Tunis, ENIT-LAMSIN, B.P. 37, 1002 Tunis, Tunisia}
\email{amal.abidi@enit.utm.tn}

\date{\today}

\begin{abstract}
We derive conditional stability estimates for inverse scattering problems related to time harmonic magnetic Schr\"odinger equation. We prove logarithmic type estimates for retrieving the magnetic (up to a gradient) and electric potentials from near field or far field maps. Our approach combines techniques from similar results obtained in the literature for inhomogeneous inverse scattering problems based on the use of geometrical optics solutions. 
\end{abstract}
\maketitle

\section{Introduction and main results}\label{S1}
\setcounter{equation}{0}
\subsection{Introduction}
This paper is concerned with the inverse scattering problem of recovering the magnetic and electric potentials in the magnetic Schr\"odinger model from near field or far field measurements at a fixed frequency. The forward model is as follows  (see for instance \cite{texbook5, texbook22b}).
Let $D\subset \R^3$ be a bounded open set with smooth boundary such that $\R^3\setminus D$ is connected and let $A=(a_1,a_2,a_3)\in W^{1,\infty}(\R^3)^3$ be a real valued vector modeling the magnetic potential and  $q\in L^\infty(\R^3)$ be a complex valued function with non negative imaginary part modeling the electric potential such that $\textrm{Supp} (A) \subset D$ and $\textrm{Supp} (q)\subset D$. The magnetic Schr\"odinger operator we are considering is 
\begin{equation}\label{1.1}
\HAV :=-(\nabla +iA)^2+ q =  -\Delta-Q_{A,q},
\end{equation}
where $Q_{A,q}$ is the first order operator given by
\begin{equation}\label{1.2}
Q_{A,q} v:=i\; \textrm{div}(Av)+iA\cdot \nabla v-(\abs{A}^2+q)v,\quad v\in H^1_{\mathrm{loc}}(\R^3).
\end{equation}

The direct scattering problem in the near field setting can be phrased as follows. Let $B$ be a smooth bounded and simply connected domain (typically a ball) containing $D$ with outward normal denoted by $\nu$ and let $y \in \partial B$ be the location of a point source.  The total field $u (\cdot, y)$ generated by the point source satisfies 
\begin{equation}\label{1.4-n}
\HAV u (\cdot, y)-k^2u(\cdot, y) = \delta_{y} \quad \mathrm{ in } \; \R^3,
\end{equation}
with $\delta_{y}$ denoting the Dirac distribution at $y$ and $k>0$ is the wave number. The total field is decomposed into 
\begin{equation}\label{1.3-n}
u_{A,q}(\cdot, y)=\Phi(\cdot, y)+u_{A,q}^s(\cdot, y)\quad \mathrm{ in } \; \R^3,
\end{equation}
where the scattered field $u_{A,q}^s(\cdot, y) \in H^2_{\mathrm{loc}}(\R^3)$ and satisfies the Sommerfeld radiation condition at infinity 
\begin{equation}\label{1.5}
\lim_{r\rightarrow \infty} r\Big( \p_r u^s -iku^s \Big)=0, \quad r=\vert x\vert
\end{equation}
uniformly with respect to $\hat{x}=\frac{x}{\vert x\vert}$. The incident field is given by
\begin{equation}\label{1.7}
\Phi(x,y):=\frac{1}{4\pi}\frac{e^{ik\abs{x-y}}}{\abs{x-y}},\quad x\neq y
\end{equation}
 and is the fundamental solution of the Helmholtz equation, i.e. satisfying \eqref{1.4-n} for $A=q=0$ together with the Sommerfeld radiation condition. 
\smallskip
 
 The first inverse problem that we shall investigate is to recover $A$ and $q$ from the knowledge  $u_{A,q}^s(x,y)$ for all $(x,y) \in \partial B\times \partial B$. Defining the  near field operator 
$\mathcal{N}_{A,q}:\; L^2(\p B) \rightarrow L^2(\p B),$ as 
 \begin{align}\label{NN}
\mathcal{N}_{A,q} h(x) :=\int_{\partial B} u^s_{A,q}(x,y) h(y) \, \ds(y), \quad x\in \p B,
\end{align}
where $u^s_{A,q}(\cdot, y)$ is given by \eqref{1.3-n} and satisfying \eqref{1.5}, the inverse problem in the near field setting can be equivalently stated as identifying  $A$ and $q$ from the knowledge of   $\mathcal{N}_{A,q}$.
\smallskip

The direct scattering problem in the far field setting formally corresponds with letting $|y|\to \infty$ in the direction $-d$ with $d\in\mathbb{S}^2$ (the unit sphere of $\R^3$) and can be phrased as follows: Given an incident plane wave $u^i(x, d)=e^{ik x\cdot d}$, $x\in\R^3$, seek a total field $u_{A,q}(\cdot, d)$ that satisfies 
\begin{equation}\label{1.4}
\HAV u (\cdot, d)-k^2u(\cdot, d) = 0 \quad \mathrm{ in } \; \R^3
\end{equation}
and can be decomposed into 
\begin{equation}\label{1.3}
u_{A,q}(\cdot, d)=u^i(\cdot, d)+u_{A,q}^s(\cdot, d) \quad \mathrm{ in } \; \R^3,
\end{equation}
where the scattered field $u_{A,q}^s(\cdot, d) \in H^2_{\mathrm{loc}}(\R^3)$ and satisfies the Sommerfeld radiation condition. The latter implies in particular that the scattered field has the following asymptotic behavior as $|x| \to \infty$,
\begin{equation}\label{1.10}
u_{A,q}^s(x,d)=\frac{e^{ik\vert x\vert}}{\vert x\vert}\Big( u_{A,q}^\infty (\hat{x},d)+O\Big( \frac{1}{\vert x\vert}\Big)\Big), \quad 
\end{equation}
where $u_{A,q}^\infty(\cdot,d)$ is the so-called far field pattern.  The second inverse problem that we shall consider is the identification of $A$ and $q$ from the knowledge of  $u_{A,q}^\infty(\hat{x}, d)$ for all $ (\hat{x}, d) \in \mathbb{S}^2 \times \mathbb{S}^2$. 
\smallskip

Our main goal for both settings is to investigate  the conditional stability of recovering  the electric potential and the magnetic field from the given so-called full aperture  measurements.   Our strategy relies  the use of geometrical optics solutions to the magnetic Schr\"odinger equation to derive  estimates on the Fourier coefficients of these potentials under additional regularity assumptions. The stability for the electric potential requires the use of an adapted version of the Helmholtz decomposition. Transferring the results from the near field operator to the far field measurements necessitate a careful study of the near field to far field mapping. The reciprocity between the operators ${\mathcal{H}_{A,q}} -k^2$ and ${\mathcal{H}_{-A,q}} -k^2$ (supplemented with the radiation condition) and mixed reciprocity between far field and near field play a central role in the proofs.  
\smallskip

In the absence of the magnetic potential $A$, the study of the identifiability  of $q$ from  full aperture measurements is one of the first foundational problems in inverse scattering theory and we refer to \cite{texbook017,texbook023, texbook026, texbook030} for pioneering uniqueness results under various (regularity) assumptions. 
In the presence of a magnetic potential $A$, we remind that there is an obstruction to uniqueness for both near field and far field settings (as has been noted in \cite{texbook9} for instance). In fact, the scattered field outside a ball $B$ containing $D$ is invariant under the gauge transformation of the magnetic potential. Namely, given $\varphi\in W^{2,\infty}(\Rt)$ with support compactly embedded in $B$ and letting $\tilde{u}=u(x)e^{-i\varphi (x)}$ one easily observes that 
\begin{equation}\label{1.12}
\mathcal{H}_{A+\nabla\varphi,q}\tilde{u}:=-(\nabla +i(A+\nabla \varphi ))^2\tilde{u}+q(x)\tilde{u}= e^{-i\varphi (x)} \mathcal{H}_{A,q}u.
\end{equation}
Since $\varphi =0$ outside $B$,  $\tilde{u}$ then satisfies the same equation as $u$, namely \eqref{1.3-n} (respectively \eqref{1.3}) in the near field setting (respectively in the far field setting) with $\mathcal{H}_{A,q}$ replaced by $\mathcal{H}_{A+\nabla\varphi,q}$. Let us denote by $u^s_{A+\nabla\varphi,q}$ the scattered field associated with the potentials $A+\nabla\varphi$ and $q$. From uniqueness of solutions to the above stated scattering problems one easily deduces  that for all $y \in \partial B$ and $d \in \s^2$
$$
u^s_{A+\nabla\varphi,q}(\cdot,y) = (e^{-i\varphi (x)} -1) \Phi(\cdot, y) + e^{-i\varphi (x)} u_{A,q}^s(\cdot,y)  \quad \mathrm{ in } \; \R^3,
$$
$$
u^s_{A+\nabla\varphi,q}(\cdot,d) = (e^{-i\varphi (x)} -1) u_i(\cdot, d) + e^{-i\varphi (x)} u_{A,q}^s(\cdot,d) \quad \mathrm{ in } \; \R^3.
$$
This clearly shows that $u^s_{A+\nabla\varphi,q}(\cdot,y)  = u^s_{A,q}(\cdot,y) $ and $u^s_{A+\nabla\varphi,q}(\cdot,d) = u^s_{A,q}(\cdot,d)$ outside $D$ and therefore, the magnetic potential $A$ cannot be uniquely determined from  far field or near field measurements outside $B$. It indicates that the best we can expect from the knowledge of the near field operator $\mathcal{N}_{A,q}$ or the far field $u_{A,q}^\infty$ is to identify $(A,q)$ modulo a gauge transformation of $A$. When $\textrm{Supp}(A)\subset D$ is known, the problem  may be equivalently reformulated as whether the magnetic field defined by the $2$-form associated with the vector $A$,
\begin{equation}\label{1.14}
\textrm{curl}\, A:=\frac{1}{2}\sum_{i,j=1}^3\para{\p_{x_j}a_i-\p_{x_i}a_j} dx_j\wedge dx_i,
\end{equation}
and the electric potential $q$ can be retrieved from far field or near field measurements. The uniqueness for similarly stated inverse problems has been established  in \cite{texbook22} for $L^\infty$ regularity of the coefficients. It has been studied in earlier works under more regularity assumptions in \cite{NSU95} and for small perturbations in \cite{texbook9,texbook5}. We also quote the recent uniqueness result in \cite{texbook5.}  for measurements associated with finite number of incident waves but with full frequency range.

Concerning stability results with full aperture measurements, in \cite{texbook3} H\"ahner and Hohage  established logarithmic stability estimates for the case $A=0$. These results improve previous ones due to Stefanov \cite{texbook26.} by giving an explicit exponent in the logarithmic estimate and using the $L^2$-norm for far field patterns. In \cite{texbook28} Isaev and Novikov proved stability estimates with explicit dependence on the wave number. 
We hereafter shall follow a similar approach as in \cite{texbook3, texbook26.} to study the case when one would like to simultaneously recover $\textrm{curl}\, A$ and $q$ from full aperture measurements in the light of geometrical optics solutions developed in \cite{texbook4, texbook25, texbook9, texbook6} for various context in relation with the inverse problem we are interested in.  We establish the stability result first for the magnetic field in the case of near field data. We then employ a carefully designed Helmholtz decomposition to infer the stability result for the electric potential. The derivation of the results for far field data are obtained after establishing some key properties relating this data to the near field data.

For bounded domains, the inverse problem with full aperture measurements  corresponds with measuring the global Dirichlet to Neumann map. For this problem Tzou proved in \cite{texbook6} $\mathrm{log}$-type stability estimate for $H^{-1}$ norms of the coefficients, assuming that the magnetic potentials are in $W^{2,\infty}$ and the electric potential are in $L^\infty$. We here consider stability with respect to the $L^\infty$ norm with explicit link between the additional needed regularity for the coefficients and the logarithm exponent.   Let us finally indicate that   uniqueness and  log-log stability results with partial data have been also studied by many authors in the literature  (see for instance \cite{texbook6, Caro-Pohjola, P-MR2018, DKSU2007, texbook4}) but are not addressed in the present work.
\subsection{Main stability results}

We here state the main results of this paper concerning conditional log-stability reconstruction of the magnetic field  $\textrm{curl}\, A$ given by (\ref{1.14}) and the electric potential $q$ from knowledge of the full aperture far field measurements, i.e., $u_{A,q}^\infty(\hat{x},d)$ for any $(\hat{x},d)\in\mathbb{S}^2\times \mathbb{S}^2$ or from knowledge of the near field operator $\mathcal{N}_{A,q}$. 
\smallskip

Let us first indicate the required conditions for admissible compactly supported magnetic potentials $A$ and electric potentials $q$. Let $M>0$ and $\sigma>0$ be given.  We define the class of admissible magnetic potentials $\mathcal{A}_\sigma(M)$ by
\begin{multline}\label{1.15}
\mathcal{A}_\sigma(M):=\{A\in W^{2,\infty}(\R^3,\R^3),\,\mathrm{Supp}(A)\subset D,\,\cr
\norm{A}_{W^{2,\infty}}\leq M,\, \textrm{and}\, \norm{\widehat{\textrm{curl} A}}_{L^1_\sigma(\R^3)}\leq M\},
\end{multline}
where $\widehat{v}$ denotes the Fourier transform of $v$ and $L^1_\tau(\R^3)$ is the weighted $L^1(\R^3)$ space with norm
$$
\norm{v}_{L^1_\tau(\R^3)}:=\int_{\R^3} (1+|\xi|^2)^{\tau/2} |v(\xi)| d\xi.
$$
Given $M>0$ and $\gamma>0$, we define the class of admissible electric potentials $\mathcal{Q}_\gamma(M)$ by
\begin{multline}\label{1.16}
\mathcal{Q}_\gamma (M):=\{q\in L^{\infty}(\R^3,\C),\,\Im(q) \ge 0, \, \mathrm{Supp}(q)\subset D,\,\cr
\norm{q}_{L^{\infty}(D)}\leq M\, \textrm{and}\, \norm{\widehat{q}}_{L^1_\gamma(\R^3)}\leq M\}.
\end{multline}
The first main result of this paper is the following log-stability for the magnetic field $\mathrm{curl}\, A$ and the electric potential $q$ from the near field measurements.
\begin{theorem}\label{T1.1}
 Let $M>0$, $\sigma>0$ and $\gamma>0$. Then there exists a constant $C>0$ such that for any  $(A_j,q_j)\in\mathcal{A}_\sigma(M)\times \mathcal{Q}_\gamma(M)$, $j=1,2$, we have
$$
\norm{\mathrm{curl}(A_1-A_2)}_{L^\infty}
\leq C \big(\norm{\mathcal{N}_{A_1,q_1}-\mathcal{N}_{A_2,q_2}}^{1/2}+\big(\log^-(\norm{\mathcal{N}_{A_1,q_1}-\mathcal{N}_{A_2,q_2}})\big)^{-\frac{\sigma}{(\sigma+3)}}\big),
$$
$$
 \norm{q_2-q_1}_{L^\infty} \leq C \big(\norm{\mathcal{N}_{A_1,q_1}-\mathcal{N}_{A_2,q_2}}^{1/2} + \big(\log^-(\norm{\mathcal{N}_{A_1,q_1}-\mathcal{N}_{A_2,q_2}})\big)^{-\frac{\gamma \sigma}{(\sigma+3)(2\gamma+3)}}   \big),
$$
where $\log^-(t):=\max(-\log(t),0)$ for $t>0$. Here $C$  depends only on $B$, $M$, $\sigma$ and $\gamma$. 
\end{theorem}
An exactly similar stability result can be deduced for far field measurements if one uses the following very restrictive norm on the measurements.  Let
\begin{equation}\label{1.17}
\Gamma:=\set{(\ell,m),\, \ell\in\N\cup\{0\},\,m\in[\![ -\ell,\ell ]\!]},
\end{equation}
 and denote by $Y_\ell^m$, $(\ell,m)\in\Gamma$ the complete system of special harmonics on $\s^2$. For a far field pattern, $u^{\infty}$ we denote by $\mu_{(\ell_1,m_1;\ell_2,m_2)}$, $(\ell_i,m_i)\in\Gamma$, $i=1,2$ its Fourier coefficients given by 
\begin{equation}\label{1.18}
\mu_{(\ell_1, m_1; \ell_2, m_2)}:= \int_{\s^2}\int_{\s^2} u^{\infty}(\hat{x},d)\overline{Y^{m_1}_{\ell_1}}(\hat{x})\overline{Y^{m_2}_{\ell_2}}(d) \; \ds(\hat{x})\; \ds(d).
\end{equation}
Let $a>0$ such that $D \subset \{x\in \R^3; |x|<a\}$. Following \cite{texbook3}, we then introduce the following norm  
\begin{equation}\label{1.19}
\Vert u^{\infty} \Vert_{\mathcal{F}}^2 := \sum_{(\ell_1,m_1)\in\Gamma} \sum_{(\ell_2,m_2)\in\Gamma} \Big( \frac{2\ell_1+1}{eka}\Big)^{2\ell_1}  \Big(\frac{2\ell_2+1}{eka}\Big)^{2\ell_2} \left\vert \mu_{(\ell_1, m_1; \ell_2, m_2)}\right\vert^2.
\end{equation}
In Lemma \ref{L3.5} below we prove that this norm is finite for all far fields $u_{A,q}^\infty$ with  $A\in\mathcal{A}_\sigma(M)$ and $q\in\mathcal{Q}_\gamma(M)$.  Using Lemma \ref{L3.6} and Theorem \ref{T1.2}, we immediately get.
\smallskip

\begin{theorem}\label{T1.2}
Let $M>0$, $\sigma>0$ and $\gamma>0$. Then there exists a constant $C>0$ such that for any  $(A_j,q_j)\in\mathcal{A}_\sigma(M)\times \mathcal{Q}_\gamma(M)$, $j=1,2$, we have
$$
\norm{\mathrm{curl}(A_1-A_2)}_{L^\infty}
\leq C \big(\norm{u^\infty_{A_1,q_1}-u^\infty_{A_2,q_2}}_{\mathcal{F}}^{1/2}+ \big(\log^-(\norm{u^\infty_{A_1,q_1}-u^\infty_{A_2,q_2}}_{\mathcal{F}})\big)^{-\frac{\sigma}{(\sigma+3)}} \big),
$$
and 
$$
 \norm{q_2-q_1}_{L^\infty} \leq C \big(\norm{u^\infty_{A_1,q_1}-u^\infty_{A_2,q_2}}_{\mathcal{F}}^{1/2} + \big(\log^-(\norm{u^\infty_{A_1,q_1}-u^\infty_{A_2,q_2}}_{\mathcal{F}}))^{-\frac{\gamma \sigma}{(\sigma+3)(2\gamma+3)}}   \big).
$$
 Here $C$ depends only on $D$, $a$, $M$, $\sigma$ and $\gamma$. 
\end{theorem}
One can also obtain a slightly modified stability result using the $L^2$ norm of the measurements following the method in \cite{texbook3}. It is summarized in the following theorem.
\begin{theorem}\label{T1.3}
Let $M>0$, $\sigma>0$, $\gamma>0$ and  $\varepsilon>0$. Then there exist two constants $C>0$ and $\delta >0$  such that for all  $(A_j,q_j)\in\mathcal{A}_\sigma(M)\times \mathcal{Q}_\gamma(M)$, $j=1,2$ verifying $\norm{u^\infty_{A_1,q_1}-u^\infty_{A_2,q_2}}_{L^2(\s^2 \times \s^2)} < \delta$ we have
\begin{equation*}
\norm{\mathrm{curl}(A_1-A_2)}_{L^\infty(D)}\leq C \big( \log^-(\norm{u^\infty_{A_1,q_1}-u^\infty_{A_2,q_2}}_{L^2(\s^2 \times \s^2)}\big)^{-\frac{\sigma}{\sigma+3}+ \epsilon},
\end{equation*}
\begin{equation*}
\norm{q_2-q_1}_{L^\infty(D)} \leq C \big(\log^-(\norm{u^\infty_{A_1,q_1}-u^\infty_{A_2,q_2}}_{L^2(\s^2 \times \s^2)}))^{-\frac{\gamma \sigma}{(\sigma+3)(2\gamma+3)}+\epsilon}.
\end{equation*}
 Here $C$ depends only on $D$, $a$, $M$, $\sigma$, $\epsilon$, $\delta$ and $\gamma$. 
\end{theorem}

From Theorems \ref{T1.1} and \ref{T1.2} (or \ref{T1.3}) we immediately derive the uniqueness corollary.
\begin{corollary}\label{Cor1}
Let $A_1$ and $A_2\in \mathcal{A}_\sigma(M)$ two vector fields, $q_1$ and $q_2\in \mathcal{Q}_\gamma(M)$ and $B\supset D$. Then, we have
$$
u^\infty_{A_1,q_1}(\hat{x},d)=u^\infty_{A_2,q_2}(\hat{x},d),\quad \forall (\hat{x},d)\in \s^2\times \s^2,
$$
or
$$
u^s_{A_1,q_1}(x,y)=u^s_{A_2,q_2}(x,y),\quad \forall (x,y)\in \p B\times \p B,
$$
implies $q_1=q_2$ and $\mathrm{curl}\, A_1=\mathrm{curl}\, A_2$ in $D$.
\end{corollary}

The remainder of this paper is organized as follows. In Section 2, we  give a brief outline of some basic properties of solutions to the Helmholtz equation, the magnetic Lippmann-Schwinger equation and some properties of the near field operator. In Section 3, we review the construction of the complex geometric optics solutions due to Tzou \cite{texbook6} and we estimate the magnetic fields and the electric potentials respectively from  the near field operator. In section 4 we establish a relation between the far field and the near field and then we prove Theorems \ref{T1.2} and \ref{T1.3}. An appendix is dedicated to some technical results  in relation with   geometric optics solutions and the connection between far field and near field data.
\section{Helmholtz equation and magnetic Lippmann-Schwinger equation}
\label{sec2}
We here outline some results on the well-posedness of the direct scattering problem using the formulation of the problem as a Lippmann-Schwinger equation (see for instance \cite{texbook5, texbook10}) and prove a uniform bound with respect to the potentials.\
\smallskip

Throughout this section, we assume that $A\in W^{1,\infty}(\R^3,\R^3)$ and $q\in L^\infty(\R^3,\C)$ with $\mathrm{Supp}(A)$, $\mathrm{Supp}(q)\subset D$ and $\Im(q)\geq 0$. 
\smallskip

The function $v$ will be referring in this section to the incident wave (i.e. $\Phi(\cdot, y)$ or $u^i(\cdot, d)$), the associated total field is denoted by $u=u_{A,q}$ and the scattered field $u^s =u^s_{A,q}= u -v \in H^2_{\mathrm{loc}}(\R^3)$. Both scattering problems can then be stated as solving for $u^s \in H^2_{\mathrm{loc}}(\R^3)$ satisfying 
\begin{equation}\label{sca}
-\Delta u^s - k^2 u^s = Q_{A,q} (u^s +v) \quad \mathrm{ in } \; \R^3,
\end{equation}
and the Sommerfeld radiation condition \eqref{1.5}. For the study of this problem, we only require that $v\in H^1(D)$.
Convolution properties imply in particular that $u^s$ can be represented as 
\begin{equation}\label{1.6}
u^s(x)= \int_{D}\Phi(x,y)Q_{A,q}u(y)\, dy,\quad x\in\Rt,\, \mbox{ with } u = u^s + v \mbox{ in } D.
\end{equation}
Let us introduce the integral operator $T_{A,q} : H^1(D) \to H^1(D)$ defined by
\begin{equation}\label{1.9}
T_{A,q}w(x):=\int_{D}\Phi(x,y)Q_{A,q}w(y)dy, \quad x\in D.
\end{equation}
We remark that since $Q_{A,q} : H^1(D) \to L^2(D)$ is continuous (by regularity assumptions on $A$ and $q$), and since the volume potential 
$$
w \mapsto \int_{D}\Phi(\cdot ,y)w(y)dy,
$$
continuously maps $L^2(D)$ into $H^2_{\mathrm{loc}}(\R^3)$ (see \cite{texbook2}), we deduce  that   $T_{A,q}$ is compact. Equation \eqref{1.6} implies in particular that the total field  $u  \in H^1(D)$ and is a solution of the Lippmann-Schwinger equation 
\begin{equation}\label{lipmann}
u - T_{A,q} u = v \quad \mbox{ in } H^1(D). 
\end{equation}
Conversely, if $u \in H^1(D)$ satisfies \eqref{lipmann}, then one easily verifies using the properties of volume potentials \cite{texbook2} that $u_{A,q}^s$ given by \eqref{1.6} is in $ H^2_{\mathrm{loc}}(\R^3)$ and is a solution of the scattering problem \eqref{sca}-\eqref{1.5}.
The well posedness of the latter is then a consequence of the following proposition.
\begin{proposition} \label{invertTAq}
The operator $I-T_{A,q} : H^1(D) \to H^1(D)$, with $I$ denoting the identity operator on  $H^1(D)$ is  continuously convertible.
\end{proposition}
\begin{proof}
 The operator $I-T_{A,q}$ is of Fredholm type with index 0. It is therefore sufficient to prove the injectivity of this operator. If $u - T_{A,q} u = 0$, then from the above equivalence, $u^s$ given by \eqref{1.6} with $v=0$, satisfies 
$$
-\Delta u^s - k^2 u^s = Q_{A,q} u^s \quad \mathrm{ in } \; \R^3.
$$
Multiplying by $\overline{u^s}$ and integrating over a ball $B$ containing $D$ implies after applying the Green's theorem in both sides
\begin{multline*}
\int_{B} ( |\nabla u^s|^2 - k^2 |u^s|^2) dx - \int_{\partial B} \partial_r  u^s \, \overline{u^s} \ds(x)\cr
 = \int_{B}\big( i A \cdot ( \nabla u^s \overline{u^s} - \nabla  \overline{u^s} u^s) - (|A|^2 + q) |u^s|^2 \big)dx.
\end{multline*}
Taking the imaginary part of the previous equality implies that $\Im ( \int_{\partial B_R} \partial_r u^s \, \overline{u^s} \ds(x) ) \ge 0$ (since $A$ is real valued and $\Im(q) \ge 0$). The Rellich lemma then implies that $u^s=0$ in $\R^3\setminus B$. We now observe that 
$$
|\Delta u^s(x)| \le (k^2 + \|q\|_\infty + \|A\|^2_\infty +  \|\nabla\cdot A\|_\infty) |u^s(x)| + 2 \|A\|_\infty |\nabla u^s(x)|,\, \mbox{ for a.e. } x \in \R^3.
$$ 
Unique continuation theorem yields  $u^s=0$ in $\R^3$ and therefore $u=0$ in $D$. This proves the injectivity of $I-T_{A,q}$ and finishes the proof of the proposition.
\end{proof}
It is also possible to prove the following uniform bound.

\begin{proposition} \label{uniformboundT}
 Let $A\in W^{1,\infty}(\R^3,\R^3)$ and $q\in L^\infty(\R^3,\C)$  as above and such that $\|A\|_{W^{1,\infty}} \le M$ and $\|q\|_{L^{\infty}} \le M$, for some constant $M >0$. Then there exists a constant $C$ that only depends on $D$, $M$ and $k$ such that $\| (I-T_{A,q})^{-1}\| \le C$. Here $\|\,\cdot\,\|$ denotes the norm in $\mathcal{L}(H^1(D))$.
\end{proposition}
\begin{proof}
We prove the result using a contradiction argument. Let us assume that, for each $n \in \mathbb{N}$, there exists $A_n \in W^{1,\infty}(\R^3)^3$ and  $q_n \in L^\infty(\R^3)$ as in the proposition such that
$$
\| (I-T_{A_n,q_n})^{-1}\| \ge  n.
$$
This implies in particular the existence of a non trivial function $v_n \in H^1(D)$ such that the function $u_n \in H^1(D)$ satisfying $u_n - T_{A_n,q_n} u_n = v_n \mbox{ in } H^1(D)$ verifies
$$
\Vert u_{n}\Vert_{H^1(D)} \geq n\Vert v_n\Vert_{H^1(D)} .
$$
This gives for the normalized sequence  $\tilde{u}_{n}=\frac{u_{n}}{\Vert u_{n}\Vert_{H^1(D)}}$ 
\begin{equation} \label{toto1}
\tilde{u}_{n} - T_{A_n,q_n}\tilde{u}_{n}=\frac{v_{n}}{\Vert u_{n}\Vert_{H^1(D)}}=:\tilde{v}_{n}\quad \textrm{in } H^1(D),
\end{equation}
where $\Vert \tilde{v}_{n}\Vert_{H^1(D)} \leq \frac{1}{n}$. The associated scattered field $\tilde{u}^s_{n}\in H^2_{\mathrm{loc}}(\R^3)$ is defined by  
\begin{equation}\label{1.6b}
\tilde{u}^s_{n}(x)= \int_{D}\Phi(x,y)Q_{A_n,q_n}\tilde{u}_{n}(y)\, dy,\quad x\in\Rt,\, \mbox{ with } \tilde{u}_{n} = \tilde{u}^s_{n} + \tilde{v}_{n} \mbox{ in } D.
\end{equation}
 Since the sequence $(\tilde{u}_{n})$ is  bounded in $H^1(D)$, the assumptions on $A_n$ and $q_n$ imply that the sequence $(Q_{A_n,q_n} (\tilde{u}_{n}))$ is also  bounded in $L^2(D)$. It yields in particular, using \eqref{1.6b}, that the sequence $(\tilde{u}^s_{n})$ is  bounded in $H^2(D)$.  Using the Rellich-Kondrachov compactness theorem, we infer that, an extracted subsequence, that we keep denoting  $(\tilde{u}^s_{n})$ is a Cauchy sequence in $H^1(D)$. From $\tilde{u}_n = \tilde{u}^s_{n} + \tilde v_n$ we deduce that $(\tilde{u}_{n})$ is also a Cauchy sequence  in $H^1(D)$ and therefore converges to some $u$ in $H^1(D)$. Given the boundedness of the sequences $(A_n)$ and $(q_n)$, by changing the original sequence   (and without corrupting the contradiction argument), one can assume that $(A_n)$ and $(q_n)$ weak-* converge to $A$ and $q$ respectively in $W^{1,\infty}(\R^3)^3$ and $L^\infty(\R^3)$. One then easily verifies that $Q_{A_n,q_n} (\tilde{u}_{n})$ weakly converges in $L^2(D)$ to $Q_{A,q} ({u})$. Consequently, from \eqref{1.9}, we get that 
$ T_{A_n,q_n}\tilde{u}_{n}$ strongly convergences to $ T_{A,q}\tilde{u}$ in $H^1(D)$. Passing to the limit in 
\eqref{toto1} implies that $u \in H^1(D)$ verifies $u - T_{A,q} u = 0$. The limits $A$ and $q$ obviously verify the hypothesis of Proposition \ref{invertTAq} and therefore $u=0$. This contradicts $\Vert u\Vert_{H^1(D)}=\Vert \tilde u_{n}\Vert_{H^1(D)}=1$.
\end{proof} 

Let us observe for later use that, thanks to \eqref{1.6},  the far field associated with the scattered wave verifying \eqref{sca} can be expressed as 
\begin{equation}\label{1.11}
 u_{A,q}^\infty (\hat{x}):=\frac{1}{4\pi}\int_{\R^3} e^{-ik\hat{x}\cdot y}Q_{A,q}u(y) dy,\quad \hat{x}\in\s^2,
\end{equation}
where $u$ is the solution of \eqref{lipmann}.
\smallskip

As a straightforward corollary of Proposition \ref{uniformboundT}, the continuity properties of volume potentials and \eqref{1.11}, we have the following  uniform estimates for $u^s$ solution of \eqref{sca}-\eqref{1.5} and associated far field.
\begin{corollary}\label{C2.3}
Let $A\in W^{1,\infty}(\R^3,\R^3)$ and $q\in L^\infty(\R^3)$ as above such that $\|A\|_{W^{1,\infty}} \le M$ and $\|q\|_{L^{\infty}} \le M$ for some constant $M >0$. Then there exists a constant $C$  that depends only on $M$, $D$ and $k$ such that 
$$
\|u_{A,q}^s\|_{H^2(D)} \le C \|v\|_{H^1(D)}  \quad \mbox{ and } \quad \|u_{A,q}^\infty\|_{L^2(\s^2)} \le C \|v\|_{H^1(D)}  
$$
 for all $v \in H^1(D)$, where  $ u_{A,q}^s \in H^2_{\mathrm{loc}}(\R^3)$ and is solution of the scattering problem \eqref{sca}-\eqref{1.5}. Moreover, for any compact $K$ there exists a constant $C$  that depends only on $M$, $D$, $K$ and $k$ such that for all $v \in H^1(D)$
$$
\|u_{A,q}^s\|_{H^2(K)} \le C \|v\|_{H^1(D)}.
$$
\end{corollary}
Armed with with above, let us define for later use the linear and continuous solution operator $\mathcal{M}_{A,q}$ by
\begin{equation}\label{Mop}
\begin{array}{ccccc}
\mathcal{M}_{A,q} & : & H^1(D) & \rightarrow & H^2_{\mathrm{loc}}(\R^3), \\
 & & v &  \mapsto & \mathcal{M}_{A,q} v := u_{A,q}^s
\end{array}
\end{equation}
where  $u_{A,q}^s$ is the solution of \eqref{sca}-\eqref{1.5}.
\smallskip

We  use this data to define the near field operator 
$\mathcal{N}_{A,q}:\; L^2(\p B) \rightarrow L^2(\p B),$ as 
 \begin{align}\label{3.1}
\mathcal{N}_{A,q} h(x) :=\int_{\partial B} u^s_{A,q}(x,y) h(y) \, \ds(y), \quad x\in \p B,
\end{align}
where $u^s_{A,q}(\cdot, y) := \mathcal{M}_{A,q}  \Phi(\cdot, y)$, $y\in\p B$. 
We first remark that 
$$
\| \mathcal{N}_{A,q}\| \le \norm{u^s_{A,q}}_{L^2(\partial B \times \partial B) },
$$
and therefore it is sufficient to study  the stability of $   \mathcal{N}_{A,q} \mapsto (A,q) $ in order to infer stability results  in terms of near field measurements. 

We second observe that, after introducing the single-layer operator $ \mathcal{S} :\; L^2(\p B) \rightarrow H^1(D)$ defined by
\begin{align}\label{Sop}
\mathcal{S} h(x) :=\int_{\partial B} \Phi(x,y) h(y) \, \ds(y), \quad x \in D,
\end{align}
one has by linearity and continuity properties of the mapping  $\mathcal{M}_{A,q}$ the following identity
\begin{equation}\label{n}
\mathcal{N}_{A,q} h = (\mathcal{M}_{A,q} \mathcal{S} h )|_{\partial B}.
\end{equation}
This equality states that $\mathcal{N}_{A,q} h$ is nothing but the near field measurements on $\partial B$ generated by an incident field $v: = \mathcal{S} h$.
\smallskip

From properties and jump relations for single-layer potential (see \cite{texbook2}) $v= \mathcal{S} h$ with density $h\in L^2(\p B)$ is defined in $\R^3$, satisfies the Helmholtz equation in $  \R^3 \setminus \p{B}$, the Sommerfeld radiation condition \eqref{1.5} and the following continuity and jump properties across $\p B$. 
\begin{equation}\label{JJ}
\begin{array}{lll}
v^{-}(x)=v^{+}(x) =v(x) & \textrm{on}\,\,\p B,\cr
\p_\nu v^{-}(x)  -\p_\nu v^{+}(x)=h(x) &  \textrm{on}\,\, \p B,
\end{array}
\end{equation}
where $v^{+}$ and $v^{-}$ respectively denote the restriction of $v$ to $\Rt \setminus  \overline{B}$ and $ B$. In order to exploit the information encoded into the identity (\ref{n}), one can easily check the following lemma by using (\ref{JJ}).
\begin{lemma}\label{L2.4}
Let $A\in W^{1,\infty}(\R^3,\R^3)$, $q\in L^\infty(\R^3)$ as above, and $h\in L^2(\p B)$. Set $v=\mathcal{S}h$ and $u^s=\mathcal{M}_{A,q} v$. 
Then the total field $u=v+u^s$ is  solution to the transmission problem
\begin{equation}\label{2.19}
\begin{array}{llll}
\HAV u(x)=k^2 u(x)&\textrm{in } \,\R^3\backslash\p B,\cr
u^{+}(x)=u^{-}(x) &\textrm{on } \,\p B,\\
\p_\nu u^{-}(x)-\p_\nu u^{+}(x)=h(x),  &\textrm{on }\,\p B\cr
\end{array}
\end{equation}
together with the Sommerfeld radiation condition \eqref{1.5}.
\end{lemma}
We finally point out the following identity that will be useful in the sequel and that can be easily derived from \eqref{1.6} and Fubini's theorem
\begin{equation}
\label{usefulForm}
\int_{\p B} \left( \mathcal{N}_{A,q}f\right) g \ds(x) = \int_D Sg \, Q_{A,q}(\mathcal{M}_{A,q} \mathcal{S} f  +  \mathcal{S} f ) dx \quad \mbox{ for all } f,g \in L^2(\p B).
\end{equation}

 We  establish now the following result that proves that the transpose operator associated with $\mathcal{N}_{A,q}$ is equal to $\mathcal{N}_{-A,q}$. In order to ease the writing, we indicate   two useful formulas that we shall use a few times.  The first one is a consequence of the Green's theorem and states that 
\begin{equation}\label{2.13}
\int_{B}\para{\HAV u_1 \, u_2-u_1\mathcal{H}_{-A,q}u_2}dx=\int_{\partial B}\para{u_1\p_r u_2-u_2\p_r u_1}\ds(x),
\end{equation}
for all $u_1, \, u_2\in H^2(B)$. The second one is a classical consequence of the Green's theorem  and the Rellich lemma and states that \cite{texbook2}
\begin{equation}\label{2.6}
\int_{\partial B}\para{u_1\p_r u_2-u_2\p_r u_1}\ds(x) = 0,
\end{equation}
for all $u_1$, $u_2 \in H^2_{\mathrm{loc}}(\R^3 \setminus {B})$ satisfying the Helmholtz equation $\Delta u +k^2u=0$ in $\R^3\setminus \overline{B}$ and the Sommerfeld radiation condition \eqref{1.5}.

\begin{lemma}\label{L3.1} Let $A\in W^{1,\infty}(\R^3,\R^3)$, $q\in L^\infty(\R^3)$ as above. Let $y,z \in \R^3 \setminus \overline D$  and set 
$$
u^s_{A,q}(\cdot, y) := \mathcal{M}_{A,q}  \Phi(\cdot, y),\quad \textrm{and}\,\, u^s_{-A,q}(\cdot, z) := \mathcal{M}_{-A,q}  \Phi(\cdot, z).
$$
 Then we have the following reciprocity relation, 
$$
 u^s_{A,q}(z, y) = u^s_{-A,q}(y, z).
$$
This reciprocity implies in particular that
$(\mathcal{N}_{A,q})^t = \mathcal{N}_{-A,q}$, i.e.,
\begin{equation}\label{3.2}
\int_{\p B} f \left(\mathcal{N}_{-A,q}g\right) \ds(x)=\int_{\p B} \left( \mathcal{N}_{A,q}f\right) g \, \ds(x)  \quad \mbox{ for all } f,g \in L^2(\p B).
\end{equation}
\end{lemma} 
\begin{proof}
From equation \eqref{1.6}
$$
u^s_{A,q}(z,y)= \int_{D}\Phi(z,t)Q_{A,q} \, (u^s_{A,q}(t,y) + \Phi(t, y)) dt.
$$
On the other hand, applying \eqref{2.13} and \eqref{2.6} to $u_1=u^s_{A,q}(\cdot,  y)$ and  $u_2=u^s_{-A,q}(\cdot, z)$
implies
$$
\int_{B}\para{\HAV u^s_{A,q}(t,y) \, u^s_{-A,q}(t,z)-u^s_{A,q}(t ,y)\mathcal{H}_{-A,q} u^s_{-A,q}(t,z)}dt=0.
$$
Using \eqref{sca} yields
$$
\int_{D}\para{Q_{A,q}  \Phi(t,y) \, u^s_{-A,q}(t,z)-u^s_{A,q}(t ,y)Q_{-A,q} \Phi(t,z)}dt=0.
$$
Using the Green's theorem we obtain (since $y, z  \in \R^3 \setminus \overline D$ and $A$ has compact support in $D$)
$$
\int_{B}\para{  \Phi(t,y) \, Q_{-A,q}u^s_{-A,q}(t,z)-Q_{A,q}u^s_{A,q}(t ,y) \Phi(t,z)}dt=0.
$$
When then conclude, since  $\Phi(z,t) = \Phi(t,z)$
$$
u^s_{A,q}(z,y)= \int_{D}\big(\Phi(t,y) \, Q_{-A,q}u^s_{-A,q}(t,z) + \Phi(t,z)Q_{A,q} \Phi(t, y)\big) dt.
$$
Applying the Green's theorem to the second term in the integral finally shows that
$$
u^s_{A,q}(z,y)= \int_{D}\Phi(t,y) \, Q_{-A,q}\big (u^s_{-A,q}(t,z) + \Phi(t,z)\big) dt = u^s_{-A,q}(y,z).
$$
Identity \eqref{3.2} is a direct consequence of the reciprocity relation and the Fubini theorem.
\end{proof}
\section{Stability analysis for near field data}
\setcounter{equation}{0}
The aim of this section is to prove the stability estimates given in Theorem \ref{T1.1}. The first step will be to use the properties of the  near fields to prove an orthogonality identity, which relates the difference of potentials to the difference of near field operators. Then we will use a special family of solutions called complex geometric optics solutions (CGO-solutions) to estimate the Fourier transform of the difference of the magnetic fields and the difference of the electric potentials.
\smallskip
  
Consider two pairs of potentials $(A_j, q_j) \in W^{1,\infty}(\R^3, \R^3) \times L^\infty(\R^3,\C)$, $j=1,2$, satisfying the same assumptions as $(A,q)$ in the beginning of the previous section. We set
\begin{equation}\label{3.3}
A(x):=(A_2-A_1)(x),\quad  q(x):=(q_2-q_1)(x),\quad x\in\R^3,
\end{equation}
and introduce  the first order operator $\mathcal{P}_{A_1,A_2,q}$ defined by
\begin{equation}\label{3.4}
\mathcal{P}_{(A_1,A_2,q)}v:=i\textrm{div}(Av)+iA\cdot\nabla v+(\abs{A_2}^2-\abs{A_1}^2+q)v,\quad v\in H^1(\R^3),
\end{equation}
here we remark that the coefficients of the first order operator $\mathcal{P}_{(A_1,A_2,q)}$ are supported in $D$.
\subsection{An orthogonality identity and a key integral inequality}
First, we present an orthogonality identity, which relates the difference of potentials to the difference of near field operators.
\begin{lemma}\label{L3.2}
 Let $f_1,\,f_2\in L^2(\p B)$ and set 
 \begin{equation}\label{wj}
u^s_1:=\mathcal{M}_{-A_1,q_1} \mathcal{S} f_1, \quad  \quad u^s_2:=\mathcal{M}_{A_2,q_2} \mathcal{S} f_2 \quad \mbox{ in } \R^3,
\end{equation}
and for $j=1, 2$,
\begin{equation}\label{vj}
v_j :=  \mathcal{S} f_j   \quad\mbox{ in } D 
\quad \mbox{ and }  \quad u_j :=  v_j + u^s_j  \quad  \mbox{ in } D.
\end{equation}
Then the following identity holds true.
\begin{equation}\label{3.5}
\int_{\p B}\para{\mathcal{N}_{A_1,q_1} f_2-\mathcal{N}_{A_2,q_2}f_2} \, f_1 \, \ds(x)=\int_{D}\mathcal{P}_{(A_1,A_2,q)}u_1 \, u_2 \,dx.
\end{equation}
\end{lemma}
\begin{proof}
Using \eqref{3.2} we  first observe that
$$
\int_{\p B}\para{\mathcal{N}_{A_1,q_1} f_2-\mathcal{N}_{A_2,q_2}f_2} \, f_1 \, \ds(x) = \int_{\p B}\para{\mathcal{N}_{-A_1,q_1} f_1} \, f_2 - \para{\mathcal{N}_{A_2,q_2}f_2}  \, f_1 \, \ds(x).
 $$
 We deduce from \eqref{usefulForm} that 
 \begin{equation}
 \label{hit1}
 \int_{\p B}\para{\mathcal{N}_{A_1,q_1} f_2-\mathcal{N}_{A_2,q_2}f_2} \, f_1 \, \ds(x) = \int_{D} (v_2 \, Q_{-A_1,q_1} u_1 - v_1 \,  Q_{A_2,q_2} u_2) \, dx.
\end{equation}
Applying  the Green's theorem and \eqref{2.6} to $u^s_1$ and  $u^s_2$
implies
$$
\int_{B}\para{ u^s_2\, (-\Delta  u^s_1 -k^2 u^s_1) -u^s_1 (-\Delta u^s_2 -k^2 u^s_2)}dx=0, 
$$
which yields according to \eqref{sca},
$$
\int_{D}\para{u^s_2\, (Q_{-A_1,q_1} u_1) -(Q_{A_2,q_2} u_2 )u^s_1} dx=0. 
$$
Adding the left hand side of  this equality to the right hand side of \eqref{hit1} shows that 
\begin{equation}
 \label{hit2}
 \int_{\p B}\para{\mathcal{N}_{A_1,q_1} f_2-\mathcal{N}_{A_2,q_2}f_2} \, f_1 \, \ds(x) = \int_{D} (u_2 \, Q_{-A_1,q_1} u_1 - u_1 \,  Q_{A_2,q_2} u_2) \, dx.
\end{equation}
The result of the lemma follows from \eqref{hit2} after integrating by parts in the right hand side and observing that
$$
Q_{-A_1,q_1} u_1 - Q_{-A_2,q_1} u_1 = \mathcal{P}_{(A_1,A_2,q)} u_1.
$$
This completes the proof.
\end{proof}
We now prove the fundamental integral inequality, which relates the difference of two magnetic potentials and electric potential in $D$ to the difference between their corresponding near pattern fields. This integral inequality will be the starting point in the proof of the stability estimate for the corresponding inverse problem.
\begin{lemma}\label{L3.3}
There is a constant $C>0$  that only depends on $B$ and $k$ such that
\begin{multline}\label{3.10}
\left| \int_{B}  [iA\cdot (u_1 \nabla u_2-u_2\nabla u_1) -(\abs{A_2}^2-\abs{A_1}^2+q)u_1u_2]dx \right|\cr
\leq  C \norm{ \mathcal{N}_{A_1,q_1}-\mathcal{N}_{A_2,q_2}}\norm{u_1}_{H^2(B)} \norm{u_2}_{H^2(B)}
\end{multline}
for all $u_1 \in H^2(B)$ satisfying  $\mathcal{H}_{-A_1,q_1} u_1=k^2 u_1$ in $B$ and all $u_2 \in H^2(B)$ satisfying $\mathcal{H}_{A_2,q_2} u_2=k^2 u_2$ in $B$.
\end{lemma}
\begin{proof}
Let $ u_j $, $j=1,2$, be given as in the lemma. Let $u^{+}_j \in H^1_{\textrm{loc}}(\R^3 \setminus \overline{B})$ be the outgoing  solution to the following exterior Dirichlet problem (see for instance \cite{texbook2})
$$ \left \{
\begin{array}{lll}
\Delta u^{+}_j+k^2u^{+}_j=0 & \mbox{ in }  \R^3 \setminus \overline{B},\\
u^{+}_j=u_j & \mbox{ on }  \partial B,
\end{array}
\right . 
$$
and $u^{+}_j$ satisfies the Sommerfeld radiation condition \eqref{1.5}.
Elliptic regularity infers that $u^{+}_j \in H^2_{\textrm{loc}}(\R^3 \setminus \overline{B})$ and in particular, by trace theorems \cite[Theorem 2.1 in Chapter 4]{lionsmagenes}, 
\begin{equation}
\label{f1e}
\Vert \partial_\nu u^{+}_j\Vert_{L^2(\partial B)}\leq \tilde C \Vert u_j\Vert_{H^2( B)},
\end{equation}
for some constant $\tilde C$ that only depends on $B$ and $k$. Let us extend the functions $u_j$ as follows 
$$u_j(x):=\left \{
\begin{array}{lll}
u^{-}_{j}(x)=u_j(x) & \textrm{if}\;  x \in B,\\
u^{+}_j(x) & \textrm{if}\;x \in  \R^3 \setminus \overline{B},
\end{array}
\right. 
$$
and set
\begin{equation}
\label{deffj}
f_j:=\p_\nu u_j^{-}-\p_\nu u_j^{+}\quad  \textrm{on} \; \p B.
\end{equation}
By trace theorems and \eqref{f1e} we have that $f_j \in L^2(\p B)$ and 
\begin{equation}\label{fj}
\| f_j \|_{L^2(\p B)}^2 \leq C \Vert u_j\Vert_{H^2( B)}^2,
\end{equation}
for some constant $C$ that only depends on $B$ and $k$.
One can easily check  that  $u_1$ and $u_2$ satisfy the following transmission problems respectively
$$
\left \{
\begin{array}{lll}
\mathcal{H}_{-A_1,q_1} u_1=k^2u_1  & \textrm{ in } \R^3\backslash\p B,\\
u^{-}_1=u^{+}_1  &\textrm{ on }  \partial B,\\
\p_\nu u^{-}_1-\p_\nu u^{+}_1=f_1  &\textrm{ on }  \partial B,
\end{array}
\right . 
\qquad;\qquad 
\left \{
\begin{array}{lll}
\HAVB u_2=k^2u_2  & \textrm{ in } \R^3\backslash \p B,\\
u^{-}_2=u^{+}_2  &\textrm{ on }  \partial B,\\
\p_\nu u^{-}_2-\p_\nu u^{+}_2=f_2 &\textrm{ on }  \partial B.
\end{array}
\right . 
$$
Moreover, $u_j^{+}$, $j=1,2$, satisfy the Sommerfeld radiation condition \eqref{1.5}.  Consider now the functions 
$$
v_j(x):=\mathcal{S}f_j= \int_{\partial B} \Phi(x,y) f_j(y) \,\ds(y) \quad x \notin \p B, \quad j=1,2.
$$
Therefore, $u^s_j:= u_j - v_j$, $j=1,2$, are the same as in Lemma \ref{L2.4} and then satisfies Lemma \ref{L3.1},  i.e they verify \eqref{wj}-\eqref{vj}. Consequently, identity (\ref{3.5}) holds, namely
$$
\int_{D}\mathcal{P}_{(A_1,A_2,q)}u_1 \, u_2 \,dx = \int_{\p B}\para{\mathcal{N}_{A_1,q_1} f_2-\mathcal{N}
_{A_2,q_2}f_2} \, f_1 \, \ds(x).
$$
In view of (\ref{3.4}), we get 
\begin{multline}\label{3.15}
\int_{B}  [iA\cdot(u_1\nabla u_2-u_2\nabla u_1) -(\abs{A_2}^2-\abs{A_1}^2+q)u_1u_2]dx \cr
= \int_{\p B}\para{\mathcal{N}_{A_1,q_1} f_2-\mathcal{N}_{A_2,q_2}f_2} \, f_1 \, \ds(x).
\end{multline}
Consequently, 
\begin{multline}\label{3.16}
\left| \int_{B}   [iA\cdot(u_1\nabla u_2-u_2\nabla u_1) -(\abs{A_2}^2-\abs{A_1}^2+q)u_1u_2]dx \right| \cr
\leq \Vert \mathcal{N}_{A_1,q_1}-\mathcal{N}_{A_2,q_2}\Vert\Vert f_1\Vert_{L^2(\p B)} \Vert f_2\Vert_{L^2(\p B)}.
\end{multline}
Identity \eqref{3.10} immediately follows from \eqref{3.16} and \eqref{fj}.
\end{proof}
\subsection{Complex geometric optics solutions-CGO}
 The main strategy of the proof of stability estimate on determining the magnetic field and the electric potential  from the near field data is the use of complex geometrical optics solutions in \eqref{3.10} to estimate the Fourier coefficients of the difference of two magnetic fields $\mathrm{curl}(A_2-A_1)$ and the difference of two potential $q_2-q_1$. We therefore first outline some known results about these special solutions extracted from the literature \cite{texbook4, texbook25, texbook9, texbook6}.
Let $\omega =\omega_1+i\omega_2$ be a vector with $\omega_1,\omega_2\in \sss$ and $\omega_1 \cdot\omega_2=0$. We define the operator $N_\omega :=\omega\cdot \nabla$. Since this operator can be interpreted as the $\overline{\partial}$ operator in the complex plane defined by $\omega_1$ and $\omega_2$ one can construct an inverse operator that can be formally defined by
\begin{equation}\label{4.2}
N^{-1}_\omega (g)(x)=\frac{1}{(2\pi)^3}\int_{\R^3} e^{-ix\cdot \xi} \Big( \frac{\hat{g}(\xi)}{\omega \cdot \xi} \Big)d\xi,
\end{equation}
for a compactly supported distribution $g$ (for instance). Some key mapping properties of $N^{-1}_\omega $ are summarized in the appendix. We remark that if $\rr \in \C^3$ satisfies $\rr\cdot \rr=0$, then $\rho = s\omega$ with $s=\frac{\vert \rr\vert}{\sqrt{2}}$ and $\omega$ is as above.  With this notation for $\rho$ we have the following Lemma where the proof can be found in \cite{texbook4} which require more regularity than $W^{1,\infty}$ for the magnetic potentials.
\begin{lemma}\label{L4.1}
Let $A \in W^{2,\infty}(D)$ and $q\in L^\infty(D)$ such that $\norm{A}_{W^{2,\infty}}\leq M$, $\norm{q}_{L^\infty}\leq M$ for some positive constant $M$, and $\mathrm{Supp}(A)$, $\mathrm{Supp}(q)\subset D$. There exists $s_0>0$ such that for any $s\geq s_0$, $\rr =s\omega$ satisfying $\rr \cdot \rr =0$, there exist complex geometrical  solutions $ u(\cdot ,\rr) \in H^2(B)$ of the form
\begin{equation*}\label{4.3}
u(x,\rr)=e^{ix\cdot \rr}(e^{i\varphi (x,\omega)}+r(x, \rr )),
\end{equation*}
to the equation $\HAV u=k^2u$ in $B$, where $\varphi (x,\omega)=N^{-1}_\omega (-\omega \cdot A)$ and
\begin{equation*} \label{4.5}
\Vert r(\cdot ,\rr)\Vert_{H^m(B)}\leq Cs^{m-1}, \quad 0\leq m\leq 2 \quad\mbox{ and } \quad\Vert u(\cdot ,\rr)\Vert_{H^2(B)}\leq Cs^{2} e^{\Lambda s},
\end{equation*}
where $C$, $\Lambda$ and $s_0$ depend only on $B$, $k$ and $M$.
\end{lemma}  

In the remainder of this section we consider two pairs of potentials $(A_j, q_j)\in W^{2,\infty} \times L^\infty$, $j=1,2$, with  $\mathrm{Supp}(A)$, $\mathrm{Supp}(q)\subset D$, $\Im(q_j)\geq 0$ and satisfying 
\begin{equation}\label{Asymp-Aq2}
\norm{A_j}_{W^{2,\infty}}\leq M,\quad \norm{q_j}_{L^\infty}\leq M,\quad j=1,2,
\end{equation}
 for some $M>0$ fixed and set as previously
\begin{equation}\label{Aq}
 A(x):=(A_2-A_1)(x),\quad  q(x):=(q_2-q_1)(x),\quad x\in\R^3.
\end{equation}
\smallskip

Let $\xi\in \Rt$, $\omega_1$, $\omega_2\in\sss$ be three mutually orthogonal vectors in $\R^3$. For each $s> \abs{\xi}/2$, let 
\begin{equation}\label{4.6}
\rho_1=s\Big(i\omega_2+\big(\frac{\xi}{2s} -\sqrt{1-\frac{\vert \xi \vert^2}{4s^2}}\omega_1 \big) \Big):=s\omega^{*}_1(s),
\end{equation}
\begin{equation}\label{4.7}
\rho_2=s\Big(-i\omega_2+\big( \frac{\xi}{2s}+\sqrt{1-\frac{\vert \xi \vert^2}{4s^2}}\omega_1 \big) \Big):=s\omega^{*}_2(s).
\end{equation}
For $s\ge s_0>0$ for some $s_0$ sufficiently large (that only depends on $B$ and $k$),  Lemma \ref{L4.1} guarantees the existence of the geometrical optics solutions: $u_1 \in H^2(B)$ verifying $\h_{-A_1,q_1}u_1=k^2 u_1$ in $B$ and $u_2\in H^2(B)$ verifying $\h_{A_2,q_2} u_2=k^2u_2$ in $B$ and such that 
\begin{equation}\label{4.8}
u_j(x)=e^{ix\cdot \rr_j}(e^{i\varphi_j (x,\omega_j^{*})}+r_j(x,\rr_j )),
\end{equation}
where $r_j(\cdot,\rr_j)$, $j=1,2$, satisfies 
\begin{equation}\label{4.9}
\Vert r_j(\cdot ,\rr_j)\Vert_{H^m(D)}\leq Cs^{m-1}, \quad 0\leq m\leq 2,
\end{equation}
and where $\varphi_1(x,\omega_1^{*})=N_{\omega_1^{*}}^{-1}(\omega_1^{*} \cdot A_1)$ and $\varphi_2(x,\omega_2^{*})=N_{\omega_2^{*}}^{-1}(-\omega_2^{*} \cdot A_2)$ are a solutions of 
\begin{equation}
\omega_1^{*}\cdot\nabla\varphi_1=\omega_1^{*}\cdot A_1,\quad \omega_2^{*}\cdot\nabla\varphi_2=-\omega_2^{*}\cdot A_2.
\end{equation}
Furthermore, according to (\ref{4.8}),  (\ref{4.6}) and (\ref{4.7}), there exist $C$ and $\Lambda>0$ such that
\begin{equation}\label{4.16}
\Vert u_1u_2\Vert_{L^1(B)} \leq C, \quad \text{and}\quad \Vert u_j\Vert_{H^2(B)} \leq Cs^2e^{\Lambda s}, \text{for }\, j=1,2.
\end{equation}
\subsection{Stability estimate for the magnetic potential}
We derive in this section a stability estimate for the magnetic fields. First, we will use Lemma \ref{L3.2} and the complex geometric solutions constructed as above to estimate the Fourier transform of the difference of the magnetic fields $\textrm{curl}\, A$. Second, we exploit the condition $\norm{\widehat{\mathrm{curl}A}}_{L^\sigma}$, $\sigma>0$, is a priori bounded to prove the stability estimate.
\begin{lemma}\label{L4.2}
Let $u_j$, $j=1,2$, be the functions given by (\ref{4.8}) and set $\omega = \omega_1+i\omega_2$. Then for any $\abs{\xi}\leq s$, we have the following identity
\begin{equation}\label{4.11}
i\int_D  A(x)\cdot \left(  u_2\nabla u_1-u_1 \nabla u_2 \right) dx=2s\int_D \overline{\omega}\cdot  A(x) e^{ix\cdot \xi }dx+\mathcal{R}(\xi,s),
\end{equation} 
with $\vert \mathcal{R}(\xi,s)\vert \leq C\seq{\xi}$, where $C$ is independent of $s$, $\xi$ and $M$, with the short notation  $\seq{\xi}:=\sqrt{|\xi|^2 + 1}$.
\end{lemma}
This Lemma is a straightforward extension and adaptation of a similar lemma in \cite{texbook4} and, for shake of completeness, we provide the proof in the Appendix.
\smallskip

 In what follows, for $A_1$ and $A_2\in W^{2,\infty}(\R^3)$ as above, we introduce the notation 
\begin{equation}\label{4.12}
a_j(x)=(A_2-A_1)(x)\cdot e_j= A(x)\cdot e_j,\quad j=1,2,3,\quad x\in\R^3,
\end{equation}
where $(e_1,e_2,e_3)$ is the canonical basis of $\R^3$ and set for $j,\ell = 1,2,3$,
\begin{equation}\label{4.13}
b_{j\ell}(x):=\frac{\p a_\ell}{\p x_j}(x)-\frac{\p a_j}{\p x_\ell}(x),\quad x\in\R^3,
\end{equation}
 the components of $\textrm{curl}\, A$ and
$$
\hat{b}_{j\ell}(\xi) = \int_{\R^3} e^{ix\cdot \xi} b_{j\ell}(x) dx,
$$ 
 the associated Fourier coefficients.  We then have following estimate for the Fourier transform of the difference of the magnetic fields.
\begin{lemma}\label{L4.3}
For any $s\geq s_0$ and $\xi \in \R^3$ such that $\vert \xi \vert \leq s$ the following estimate holds true,
\begin{equation}\label{4.14}
|\hat{b}_{j\ell}(\xi)| \leq C \seq{\xi}\left( e^{\Lambda s}\data+s^{-1}\seq{\xi}\right),
\end{equation}
for $j,\ell = 1,2,3$, where $C$ and $\Lambda$ are positive constants independent of $s$, $\xi$ and $M$.
\end{lemma}
\begin{proof}
Let $\xi\in\R^3$ such that  $\vert \xi \vert \leq s$. Let $\omega=\omega_1+i\omega_2$, where $\omega_j$, $j=1,2$ are as above and  consider $u_j$, $j=1,2$ the solutions given by (\ref{4.8}). By using (\ref{3.10}) and (\ref{4.11}), we get for $\abs{\xi}\leq s$
\begin{multline}\label{4.15}
2s\big|\int_{B} \overline{\omega}\cdot  A e^{ix\cdot\xi}dx\big|\cr
\leq C \Big(\data\norm{u_1}_{H^2(B)} \norm{u_2}_{H^2(B)}
+\norm{u_1u_2}_{L^1(B)}+ \seq{\xi}\Big).
\end{multline}
Then, we obtain by (\ref{4.16})
\begin{equation}\label{4.17}
\big| \int_{B} \overline{\omega}\cdot  A(x) e^{ix\cdot\xi}dx\big|\leq C \para{ e^{\Lambda s}\data+s^{-1}\seq{\xi}}.
\end{equation}
The reasoning above remains valid if we change $\omega_1$ by  $-\omega_1$ and therefore we also have 
\begin{equation}\label{4.18}
\big| \int_{B} -\omega\cdot A(x) e^{ix\cdot\xi}dx\big|\leq C \para{ e^{\Lambda s}\data+s^{-1}\seq{\xi}}.
\end{equation}
Outside the plane ${\xi_je_k-\xi_\ell e_j} =0$, we choose $\omega_2=\frac{\xi_je_\ell-\xi_\ell e_j}{\abs{\xi_je_\ell-\xi_\ell e_j}}$ which is indeed an orthogonal unitary direction to $\xi$. Multiplying both sides of  (\ref{4.17}) and (\ref{4.18}) by $\abs{\xi_je_\ell-\xi_\ell e_j}$,  adding them together  and using the triangular inequality to eliminate $\omega_1$ imply 
\begin{equation}\label{4.19}
\big|\int_{\R^3} e^{ix\cdot\xi}(\xi_j a_\ell(x)-\xi_\ell a_j(x))dx\big|\leq C \seq{\xi}\para{ e^{\Lambda s}\data+s^{-1}\seq{\xi}},
\end{equation}
for ${\xi_je_\ell-\xi_\ell e_j} \neq0$. The inequality extends to  all $|\xi|\le s$ by regularity of both sides in terms of $\xi$ and proves \eqref{4.14}.\\
This end the proof.
\end{proof}

We are now in position to prove the main stability result for the magnetic potential from the near field data using  (\ref{4.14}) and regularity assumptions to estimate Fourier coefficients for large $|\xi|$. More precisely, we assume now that for $j=1,2$, $A_j\in W^{2,\infty}(D)$ with $\norm{A}_{W^{2,\infty}(D)}\leq M$,  and 
\begin{equation}\label{assumpAj}
\int_{\R^3} \seq{\xi}^\sigma| \widehat{\mathrm{curl} A_j}(\xi)| \, d\xi < M
\end{equation}
for some $\sigma>0$, where $ \widehat{\mathrm{curl} A_j}$ denotes the Fourier transform of $\mathrm{curl} A_j$.
\subsubsection*{End of the proof of the stability for the magnetic field}
We derive now the stability estimate for the magnetic fields in $L^\infty$-norm. Let $s_0>1$ be as in Lemma \ref{L4.3} and $s$ and $R$ be two parameters satisfying $s \ge R \ge s_0$.
From (\ref{4.14}) we get 
\begin{align*}
\int_{\R^3} | \hat{b}_{j\ell}(\xi)| \, d\xi &= \int_{\seq{\xi}\leq R}| \hat{b}_{j\ell}(\xi)| \,d\xi+\int_{\seq{\xi}\geq R}| \hat{b}_{j\ell}(\xi)| \, d\xi\\
&\leq C R^2 \left( e^{\Lambda s}\data+s^{-1} R\right) + 2M R^{-\sigma}.
\end{align*}
Choosing $R= s^{1/(\sigma+3)}$ we deduce that for $s_0$ sufficiently large (depending only on $B$, $k$, $M$ and $\sigma$),
\begin{equation}\label{4.21}
\norm{b_{j\ell}}_{L^\infty(\R^3)}\leq C'\big(e^{\Lambda' s}\data+s^{-\sigma/{\sigma+3}} \big),
\end{equation}
for some positive constants $C'$ and $\Lambda'$ and all $s \ge s_0$.  Now if $\data\leq\varepsilon_0$, for some $\varepsilon_0>0$, such that $-\log(\epsilon_0)\ge 2 \Lambda' s_0$, then taking $s = \frac{-1}{2 \Lambda'}\log(\data)$ in \eqref{4.21} implies 
\begin{equation}\label{4.22}
\norm{b_{j\ell}}_{L^\infty(\R^3)} \leq C' \big(\data^{1/2} + \big(\frac{-1}{2 \Lambda'}\log(\data)\big)^{-\sigma/{\sigma+3}} \big).
\end{equation}
We also observe that this type of inequality holds true if $\data\geq\varepsilon_0$ since in that case  we can simply write
\begin{equation}
\norm{b_{j\ell}}_{L^\infty(\R^3)} \leq M \le (M/\sqrt{\epsilon_0}) \data^{1/2}.
\end{equation}
The proof of the first estimate of Theorem \ref{T1.1} is then completed.
\smallskip

Using the above result, we are able to prove the second main result related to the stability for  the electric potential.
\subsection{Stability estimate for the electric potential}
In this section, we are going to use the complex geometric optics solutions and the stability estimate we already obtained for the {magnetic} field in order to retrieve a stability result for the electric potential. There are, however, some difficulties with this. In fact, in order to isolate the integral of the difference $q=q_2-q_1$ we would need to control the norm of the difference $A=A_2-A_1$. Unfortunately, we can only estimate the difference of the magnetic fields $\mathrm{curl}(A)$. To overcome this difficulty we will use the Helmholtz decomposition and write $A=H-\nabla\vartheta$, with the fact that $\mathrm{div}\,H=0$ and we are able to estimate the norm of $\nabla\vartheta$.
\smallskip


\begin{lemma}\label{L.AD1}
Let  $p>3$. There exist $\vartheta\in W^{3,p}(B)\cap H^1_0(B)$ and a positive constant $C$ such that
\begin{equation}\label{AD2}
\norm{\vartheta}_{W^{3,p}(B)}\leq C\norm{A}_{W^{2,p}(D)},
\end{equation}
and 
\begin{equation}\label{AD3}
\norm{A+\nabla\vartheta}_{W^{1,p}(B)}\leq C\norm{\mathrm{curl}(A)}_{L^p(D)}.
\end{equation}
Moreover, if $B'$ is a ball containing $\overline{D}$ and such that $\overline{B'}\subset B$, then
\begin{equation}\label{AD4}
\norm{\vartheta}_{W^{2,p}(B \backslash B')}\leq C\norm{\mathrm{curl}(A)}_{L^p(D)}.
\end{equation}
\end{lemma}
\begin{proof}
Let $\vartheta$ solves the following elliptic boundary value problem in the ball $B$
\begin{equation}\label{AD5}
\left\{\begin{array}{lll}
-\Delta\vartheta  &= \mathrm{div}(A) & \textrm{in}\,B,\cr
\vartheta &=  0 & \textrm{on}\,\p B.
\end{array}
\right.
\end{equation}
Since the source term $\mathrm{div}(A)$ belongs to $W^{1,p}(B)$ by the elliptic regularity (see \cite[Theorem 2.5.1.1 in Chapter 2]{Grisv}), we have $\vartheta\in W^{3,p}(B)\cap H^1_0(B)$. Moreover, there exist $C>0$ such that
\begin{equation}\label{AD6}
\norm{\vartheta}_{W^{3,p}(B)}\leq C\norm{\mathrm{div}(A)}_{W^{1,p}(D)}\leq  C\norm{A}_{W^{2,p}(D)},
\end{equation}
this ends the prove of (\ref{AD2}). To prove (\ref{AD3}), we consider the vector field $H\in W^{2,p}(B)$ defined by
\begin{equation}\label{AD7}
H=A+\nabla\vartheta.
\end{equation}
By (\ref{AD5}) and (\ref{AD7}), $H$ satisfies 
\begin{equation}\label{AD8}
\mathrm{div}(H)=0,\quad \mathrm{curl}(H)=\mathrm{curl}(A)\quad\textrm{in}\,B,\quad \textrm{and}\quad H\wedge\nu=0\quad\textrm{on}\,\p B.
\end{equation}
By the $L^p$-$\mathrm{div}$-$\mathrm{curl}$ estimate, we get
\begin{equation}\label{AD9}
\norm{H}_{W^{1,p}(B)}\leq C\norm{\mathrm{curl}(H)}_{L^p(B)},
\end{equation}
and we conclude (\ref{AD3}). Finally, to prove (\ref{AD4}) we consider a cut-off function $\chi_0\in C^\infty(\R^n)$ such that $\chi_0=1$ in $B\backslash B'$ and $\chi_0=0$ in $\overline{D}$. Set $\vartheta_0=\chi_0\vartheta$, we have by (\ref{AD5})
\begin{equation}\label{AD10}
\left\{\begin{array}{lll}
-\Delta\vartheta_0=[\Delta,\chi_0]\vartheta & \textrm{in}\, B,\cr
\vartheta_0=0 & \textrm{on}\,\p B,
\end{array}
\right.
\end{equation}
where we have used $A=0$ outside $D$. Thus, we obtain
\begin{equation}\label{AD11}
\norm{\vartheta_0}_{W^{2,p}(B)}\leq C\norm{\nabla\vartheta}_{L^p(B\backslash D)},
\end{equation}
since the first order operator $ [\Delta,\chi_0]$ is supported in $B'\backslash D$. Then we deduce that
\begin{equation}
\norm{\vartheta_0}_{W^{2,p}(B)}\leq C\norm{A+\nabla\vartheta}_{L^p(B\backslash D)}\leq C\norm{\mathrm{curl}(A)}_{L^p(D)}.
\end{equation}
This ends the proof.
\end{proof}
Applying now Morrey's inequality, we obtain for some positive constant $C$
\begin{equation}\label{AD12}
\norm{A+\nabla\vartheta}_{L^\infty(B)}\leq C\norm{\mathrm{curl}(A)}_{L^\infty(D)},
\end{equation}
and
\begin{equation}\label{AD13}
\norm{\vartheta}_{W^{1,\infty}(B \backslash B')}\leq C\norm{\mathrm{curl}(A)}_{L^\infty(D)}.
\end{equation}
We prove now the following inequality which is a slight modification of the previous estimate (\ref{3.10}) given in Lemma \ref{L3.3}. This inequality is based on the invariance of the near field operator under gauge transformation for the magnetic potential that was explained in the introduction.
\begin{lemma}\label{L.AD14}
Let $B$ denote an open ball containing $\overline{D}$ and let $\varphi\in W^{2,\infty}(B)$ with $\mathrm{Supp}(\varphi)\subset B$. Then there exist a constant $C$ that only depends on $B$ and such that
\begin{multline}\label{AD15}
\abs{\int_B e^{-i\varphi}\big[ i(A+\nabla\varphi)\cdot (u_1\nabla u_2 -u_2\nabla u_1)-((A+\nabla\varphi)\cdot(A_1+A_2)+q)u_1u_2\big] dx}\\
\leq C\norm{\mathcal{N}_{A_1,q_1}-\mathcal{N}_{A_2,q_2}}\norm{u_1}_{H^2(B)}\norm{u_2}_{H^2(B)},
\end{multline}
for all $u_1\in H^2(B)$ satisfying $\mathcal{H}_{-A_1,q_1}u_1=k^2u_1$ in $B$ and all $u_2\in H^2(B)$ satisfying $\mathcal{H}_{A_2,q_2}u_2=k^2u_2$ in the ball $B$.
\end{lemma}
\begin{proof}
Define $\tilde{A}_1=A_1-\frac{1}{2}\nabla\varphi$ and $\tilde{A}_2=A_2+\frac{1}{2}\nabla\varphi$. Consider $u_1$, $u_2\in H^2(B)$ as in the lemma. Let denote $\tilde{u}_j=e^{-i\varphi/2}u_j$, then by (\ref{1.12}) we get
\begin{align}\label{AD16}
\mathcal{H}_{-\tilde{A}_1,q_1}\tilde{u}_1 &=e^{-i\varphi/2}\mathcal{H}_{-A_1,q_1}u_1=k^2\tilde{u}_1,\cr
 \mathcal{H}_{\tilde{A}_2,q_2}\tilde{u}_2 &=e^{-i\varphi/2}\mathcal{H}_{A_2,q_2}u_1=k^2\tilde{u}_2,\quad \textrm{in}\,B.
\end{align}
Moreover, due the gauge invariance of the scattered field and since $\varphi_{|\p B}=0$, we have from (\ref{3.1}) that
\begin{equation}\label{AD17}
\mathcal{N}_{\tilde{A}_j,q_j}=\mathcal{N}_{A_j,q_j},\quad j=1,2.
\end{equation}
With these solution at hand, we use the gauge invariance (\ref{AD17}) and (\ref{3.10}) with $\tilde{A}_j$ and $\tilde{u}_j$ instead of $A_j$ and $u_j$, respectively, to obtain
\begin{multline}\label{AD18}
\left| \int_{B}  [i\tilde{A}\cdot (\tilde{u}_1 \nabla \tilde{u}_2-\tilde{u}_2\nabla \tilde{u}_1) -(|\tilde{A}_2|^2-|\tilde{A}_1|^2+q)\tilde{u}_1\tilde{u}_2]dx \right|\cr
\leq  C \norm{ \mathcal{N}_{A_1,q_1}-\mathcal{N}_{A_2,q_2}}\norm{\tilde{u}_1}_{H^2(B)} \norm{\tilde{u}_2}_{H^2(B)}.
\end{multline}
From the identity (\ref{AD18}), we get (\ref{AD15}). This ends the proof.
\end{proof}
Lemma \ref{L.AD14}, allows us then to obtain an estimate to $A$ added with a gradient term. By adding $\nabla\vartheta$, we would thus get an estimate with controlled terms. Unfortunately, we cannot directly add $\nabla\vartheta$, because of the requirement that $\mathrm{Supp}(\varphi)\subset B$ in Lemma \ref{L.AD14}. We can solve this difficulty by using a cutoff argument.
\smallskip
 
Now, we will fix $\varphi=\chi\vartheta$, for $\chi\in C_0^\infty(B)$ such that $\chi=1$ in $\overline{B}'$ and $\vartheta\in W^{2,\infty}(B)$ given by Lemma \ref{L.AD1} which satisfying (\ref{AD12}) and (\ref{AD13}).
\begin{lemma}\label{L.AD19}
 Let $u_j$, $j=1,2$ be the functions  given by (\ref{4.8}) for some $s_0>0$.  Then there exist a positive constants $C$ and $\Lambda$ such that 
\begin{equation}\label{AD20}
\abs{\int_B e^{-i\varphi}q(x)u_1u_2dx}\leq C e^{2\Lambda s}\norm{\mathcal{N}_{A_1,q_1}-\mathcal{N}_{A_2,q_2}}+s \norm{\mathrm{curl}(A)}_{L^\infty(D)}
\end{equation}
for all $s>s_0$.
\end{lemma}
\begin{proof}
 Let $u_j$, $j=1,2$ be the functions  given by (\ref{4.8}) for some $s_0>0$.  Adding and subtracting the same terms we get
\begin{align}\label{AD21}
\int_B e^{-i\varphi}q(x)u_1u_2dx =&\int e^{-i\varphi}(A+\nabla\vartheta)\cdot(A_1+A_2)u_1u_2dx\cr
&+\int_B ie^{-i\varphi}(A+\nabla\vartheta)\cdot(u_1\nabla u_2-u_2\nabla u_1)dx\cr
&+\int_B e^{-i\varphi}(A+\nabla\vartheta)\cdot\nabla\varphi u_1u_2dx+\mathcal{R}\cr
:=&\mathcal{J}_1+\mathcal{J}_2+\mathcal{J}_3+\mathcal{R},
\end{align}
where $\mathcal{R}$ denotes the integral
\begin{align*}
\mathcal{R} = -\int_B & e^{-i\varphi}\big[i(A+\nabla\vartheta)\cdot(u_1\nabla u_2-u_2\nabla u_1)\\
&\qquad +\Big((A+\nabla\vartheta)\cdot(A_1+A_2)-q+(A+\nabla\vartheta)\cdot\nabla\varphi\Big) u_1u_2\big] \, dx.
\end{align*}
First, we have
\begin{align}\label{AD22}
\abs{\mathcal{J}_2}\leq C\norm{A+\nabla\vartheta}_{L^\infty(B)}(\norm{u_1\nabla u_2}_{L^1(B)}+\norm{u_2\nabla u_1}_{L^1(B)}).
\end{align}
Then, we deduce from (\ref{AD12}) and (\ref{4.16}) that
\begin{equation}\label{AD23}
\abs{\mathcal{J}_2}\leq C s \norm{\mathrm{curl}(A)}_{L^\infty(D)}.
\end{equation}
Similarly, we have
\begin{equation}\label{AD24}
\abs{\mathcal{J}_1}\leq C\norm{A+\nabla\vartheta}_{L^\infty(B)}\norm{u_1u_2}_{L^1(B)}
\end{equation}
and then, we have
\begin{equation}\label{AD25}
\abs{\mathcal{J}_1}\leq C \norm{\mathrm{curl}(A)}_{L^\infty(D)}.
\end{equation}
Finally, by the same arguments, we have
\begin{equation}\label{AD26}
\abs{\mathcal{J}_3}\leq C \norm{\mathrm{curl}(A)}_{L^\infty(D)}.
\end{equation}
Now recall that $\varphi=\chi\vartheta$ and set $\tilde{\varphi}=(1-\chi)\vartheta$. Since $\nabla\vartheta=\nabla\varphi+\nabla\tilde{\varphi}$ we obtain, by Lemma \ref{L.AD14}, that
\begin{multline}\label{AD27}
\abs{\mathcal{R}}\leq C\norm{\mathcal{N}_{A_1,q_1}-\mathcal{N}_{A_2,q_2}}\norm{u_1}_{H^2(B)}\norm{u_2}_{H^2(B)}\cr
+\abs{\int_B ie^{-i\varphi}\nabla\tilde{\varphi}\cdot(u_1\nabla u_2-u_2\nabla u_1)+e^{-i\varphi}(\nabla\varphi+A_1+A_2)\cdot\nabla\tilde{\varphi}\, u_1u_2dx}.
\end{multline}
By the same arguments used previously, we get
\begin{equation}\label{AD28}
\abs{\mathcal{R}}\leq C e^{2\Lambda s}\norm{\mathcal{N}_{A_1,q_1}-\mathcal{N}_{A_2,q_2}}+s\norm{\nabla\tilde{\varphi}}_{L^\infty(B)}.
\end{equation}
Since by (\ref{AD13}) we have
\begin{equation}\label{AD29}
\norm{\nabla\tilde{\varphi}}_{L^\infty(B)}\leq \norm{\nabla\vartheta}_{L^\infty(B\backslash B')}\leq C \norm{\mathrm{curl}(A)}_{L^\infty(D)},
\end{equation}
we conclude (\ref{AD20}) from (\ref{AD22}), (\ref{AD25}), (\ref{AD26}) and (\ref{AD29}).
\end{proof}
We now  state the following integral identity for the electric potential which is proved in the Appendix.
\begin{lemma}\label{L4.4}
 Let $u_j$, $j=1,2$ be the functions  given by (\ref{4.8}) for some $s_0>0$.  Then for all $\xi\in\mathbb{R}^3$ and $s\geq\textrm{max}(s_0,\vert\xi\vert/2)$, we have the following identity
\begin{equation}\label{4.28}
\int_D e^{-i\varphi}q(x)u_1u_2 \,dx=\int_D q(x)e^{ix\cdot\xi}\, dx+\mathcal{R}'(\xi,s),
\end{equation}
where $\mathcal{R}'(\xi,s)$, satisfy
\begin{equation}\label{4.29}
\abs{\mathcal{R}'(\xi,s)}\leq C\big(\norm{\mathrm{curl}(A)}_{L^\infty(D)}+s^{-1}\seq{\xi}\big).
\end{equation}
The constants $C$ and $s_0$ depend only on $B$, $M$ and $k$. 
\end{lemma}
This identity allows to obtain the following estimates for the Fourier coefficients 
$$
\hat{q}(\xi) := \int_{\R^3} e^{ix\cdot \xi}  q(x) \, dx.
$$
From Lemma \ref{L.AD19} and Lemma \ref{L4.4} we deduce the following estimate.
\begin{lemma}\label{L4.5}
There exists $s_0>0$ such that for all $s\geq s_0$ and $\xi \in \R^3$ with $\vert \xi \vert \leq s$ the following estimate holds true,
\begin{equation}\label{4.30}
\vert\widehat{q}(\xi)\vert\leq C\big(e^{\Lambda s}\data+s\norm{\mathrm{curl}(A)}_{L^\infty(D)}+s^{-1}\seq{\xi}\big).
\end{equation}
The constants $s_0$, $C$ and $\Lambda$ depend only on $B$, $M$ and $k$.
\end{lemma} 
With the help of the previous lemma, we are now in position to prove the stability result for the electric potential under the assumption
\begin{equation}
\label{assumpqj}
\int_{\R^3} \seq{\xi}^\gamma| \hat{q}_j(\xi)| \, d\xi < M,
\end{equation}
for some $\gamma >0$.
\subsubsection*{End of the proof of the stability estimate for the electric potential}
Let $s_0>1$ be as in Lemma \ref{L4.5} and $s$ and $R$ be two parameters satisfying $s \ge R \ge s_0$.
From (\ref{4.30}) and  (\ref{assumpqj}) we get
\begin{align*}
\int_{\R^3} | \hat q(\xi)| \, d\xi &= \int_{\seq{\xi}\leq R} | \hat q(\xi)| ,d\xi+\int_{\seq{\xi}\geq R} | \hat q(\xi)|  \, d\xi\\
&\leq C R \left( e^{\Lambda s}\data+ s \norm{\textrm{curl}(A)}_{L^\infty}+ R s^{-1} \right) + 2M R^{-\gamma}.
\end{align*}
Choosing $R= s^{1/(\gamma+2)}$, we deduce that for $s_0$ sufficiently large (depending only on $B$, $k$ $M$ and $\gamma$),
\begin{equation}\label{4.21b}
\norm{q}_{L^\infty(\R^3)}\leq C'\big(e^{\Lambda' s}\data+s^{(\gamma+3)/(\gamma+2)} \norm{\textrm{curl}(A)}_{L^\infty} + s^{-\gamma/{(\gamma+2)}} \big)
\end{equation}
for some positive constants $C'$ and $\Lambda'$ and all $s \ge s_0$.   Observe now that \eqref{4.21} implies in particular (after eventually changing the constants $C'$, $\Lambda'$ and  $s_0$)
\begin{equation}\label{4.21c}
\norm{\textrm{curl} A}_{L^\infty(D)} \leq C'\big(e^{\Lambda' s^\kappa}\data+s^{-\kappa \sigma/(\sigma+3)}\big),
\end{equation}
for all $\kappa \ge 1$. Choosing $\kappa$ such that 
$$
-\frac{\kappa \sigma}{\sigma+3} + \frac{\gamma+3}{\gamma+2} = -\frac{\gamma}{\gamma+2} \Leftrightarrow \kappa = \frac{(2\gamma+3) (\sigma+3)}{\sigma (\gamma +2)},
$$ 
we obtain by substituting \eqref{4.21c} in  \eqref{4.21b} 
\begin{equation}\label{4.21bb}
\norm{ q}_{L^\infty(\R^3)}\leq C'\big(e^{\Lambda' s^\kappa}\data+ s^{-\gamma/{(\gamma+2)}} \big),
\end{equation}
with possibly different constants $C'$ and $\Lambda'$.

Now if $\data\leq\varepsilon_0$, for some $\varepsilon_0>0$, such that $-\log(\epsilon_0)\ge 2 \Lambda' s_0^\kappa$, then taking $s^\kappa = \frac{-1}{2 \Lambda'}\log(\data)$ in \eqref{4.21} implies 
\begin{align}\label{4.22b}
\norm{ q}_{L^\infty(\R^3)} \leq C' \big(\data^{1/2} + \big(\frac{-1}{2 \Lambda'}\log(\data)\big)^{-\gamma/{\kappa (\gamma+2)}} \big).
\end{align}
We also observe that this type of inequality holds true if $\data\geq\varepsilon_0$ since in that case  we can simply write
\begin{equation}\label{4.23b}
\norm{ q}_{L^\infty(\R^3)} \leq M \le (M/\sqrt{\epsilon_0}) \data^{1/2}.
\end{equation}
The proof of the second part of  Theorem \ref{T1.2} is then completed.

\section{Stability analysis for far field data}
\setcounter{equation}{0}

In order to prove stability estimates with far field measurements we shall exploit the relation between the far field $ u_{A,q}^{\infty} (\hat{x}, d) $, $ (\hat{x},d) \in \s^2 \times \s^2$ and the operator $\mathcal{N}_{A,q}$ given by (\ref{NN}) with $B=\{x \in \Rt, |x|<a\}$ for some sufficiently large $a>0$ so that $D \subset B$. We here assume that the hypothesis of Section 2 hold.  We recall that the far field pattern can be expressed as 
\begin{equation}\label{3.18}
u_{A,q}^{\infty}(\hat{x},d)=\frac{1}{4\pi} \int_D e^{-ik\hat{x}\cdot y} Q_{A,q}u_{A,q}(y,d)dy.
\end{equation}
 We denote by $\mu_{(\ell_1,m_1;\ell_2,m_2)}$, $(\ell_i,m_i)\in\Gamma$, $i=1,2$ the Fourier coefficients of $ u_{A,q}^{\infty}$ given by 
\begin{equation}\label{3.17}
\mu_{(\ell_1, m_1; \ell_2, m_2)}:= \int_{\s^2}\int_{\s^2} u_{A,q}^{\infty}(\hat{x},d)\overline{Y^{m_1}_{\ell_1}}(\hat{x}) \overline{Y^{m_2}_{\ell_2}}(d) \ds(\hat{x})\; \ds(d),
\end{equation}
For proving the first lemma we need the following well known results about the asymptotic of spherical Bessel functions $j_\ell$ and spherical Hankel functions of the first kind $h_\ell^{(1)}$ \cite[Appendix A]{texbook26.}: 
\begin{equation}\label{3.19}
\abs{j_\ell(kr)}\leq \alpha \Big(\frac{ekr}{2\ell+1}\Big)^\ell\frac{1}{2\ell+1},\quad 0\leq r\leq a,\quad \ell\in\N\cup\{0\},
\end{equation}
and 
\begin{equation}\label{3.20}
\abs{h_\ell^{(1)}(kr)}\leq \alpha \Big(\frac{2\ell+1 }{ekr}\Big)^\ell,\quad 0< r\leq a,\quad \ell\in\N\cup\{0\},
\end{equation}
where $\alpha$ is a constant that only depend on $a$ and $k$. We also recall the following equality that comes from the addition formula \cite{texbook2},
\begin{equation}\label{3.21}
\int_{\s^2} Y^{m_2}_{\ell_2}(\hat{z}) \Phi (x,r\hat{z})\ds(\hat{z})=ikj_{\ell_2}(kr)h^{(1)}_{\ell_2}(k\vert x\vert)Y^{m_2}_{\ell_2}(\hat{x}), \quad \vert x\vert >r,
\end{equation}
together with the Funk-Hecke formula
\begin{equation}\label{3.22}
\int_{\s^2}e^{-ikx\cdot \hat{z}} Y^{m_2}_{\ell_2}(\hat{z})\ds(\hat{z})=\frac{4\pi}{i^{\ell_2}} j_{\ell_2}(k\vert x\vert)Y^{m_2}_{\ell_2}(\hat{x}),\hspace*{0.5cm} x\in \R^3.
\end{equation}

\begin{lemma}\label{L3.5}
Let $A\in W^{1,\infty}(\R^3, \R^3)$ and $q\in L^\infty(\R^3,\C)$ be as in Section 2 and such that $\|A\|_{W^{1,\infty}} \le M$ and $\|q\|_{L^{\infty}} \le M$ for some constant $M >0$.  Let $\mu_{(\ell_1, m_1; \ell_2, m_2)}$ denote the Fourier coefficients of the far field patterns $u_{A,q}^\infty$ as defined in (\ref{3.17}). Then there exists a constant $C>0$ that only depends on $D$, $a$, $k$ and $M$ such that
\begin{equation}\label{3.23}
\abs{\mu_{(\ell_1, m_1; \ell_2, m_2)}}^2\leq C\Big(\frac{eka}{2\ell_1+1}\Big)^{2\ell_1+3}\Big(\frac{eka}{2\ell_2+1}\Big)^{2\ell_2+3}
\end{equation}
and  
\begin{equation*}\label{3.24}
 \sum_{(\ell_1,m_1)\in\Gamma}\sum_{(\ell_2,m_2)\in\Gamma} \Big(\frac{2\ell_1+1}{eka}\Big)^{2\ell_1+1}\Big(\frac{2\ell_2+1}{eka}\Big)^{2\ell_2+1}\abs{\mu_{(\ell_1, m_1; \ell_2, m_2)}}^2\leq C.
\end{equation*}
\end{lemma}
\begin{proof}
We only need to prove \eqref{3.23}. According to (\ref{3.17}) and (\ref{3.18}), we obtain
\begin{align}\label{3.25}
\mu_{(\ell_1, m_1; \ell_2, m_2)} &=\frac{1}{4\pi}\int_{B}\big(\int_{\s^2} Q_{A,q}u(y,d)\overline{Y^{m_2}_{\ell_2}}(d)\ds(d)\big)\big(\int_{\s^2} e^{-ik\hat{x}\cdot y}\overline{Y^{m_1}_{\ell_1}}(\hat{x}) \ds(\hat{x})\big)dy\cr
&:=\frac{1}{4\pi}\int_{B} w_{\ell_2,m_2}(y)v_{\ell_1,m_1}(y)dy.
\end{align}
With the help of the Funk-Hecke formula (\ref{3.22}) we compute
\begin{equation}\label{3.26}
v_{\ell_1,m_1}(y):=\int_{\s^2}e^{-ik\hat{x}\cdot y}\overline{Y^{m_1}_{\ell_1}}(\hat{x}) \ds(\hat{x})=\frac{4\pi}{i^{\ell_1}}j_{\ell_1}(k\abs{y})\overline{Y_{\ell_1}^{m_1}}(\hat{y}),\quad y\in\R^3.
\end{equation}
Then by (\ref{3.19}), we get
\begin{equation}\label{3.27}
\norm{v_{\ell_1,m_1}}_{L^2(B)}^2\leq C\int_0^a\abs{j_{\ell_1}(kr)}^2r^2dr\leq C\Big(\frac{eka }{2\ell_1+1}\Big)^{2\ell_1+3}.
\end{equation}
Using again the Funk-Hecke formula, we obtain
\begin{equation}\label{3.28}
\int_{\s^2} u(y,d)\overline{Y^{m_2}_{\ell_2}}(d)\ds(d)=T_{A,q}\Big(\int_{\s^2} u(y,d)\overline{Y^{m_2}_{\ell_2}}(d)\ds(d)\Big)+(-1)^{\ell_2}v_{\ell_2,m_2},
\end{equation}
and we deduce that
\begin{equation}\label{3.29}
w_{\ell_2,m_2}(x)=(-1)^{\ell_2}Q_{A,q}\para{(I-T_{A,q})^{-1}v_{\ell_2,m_2}}(x).
\end{equation}
Using the fact that $Q_{A,q}$ is a first order operator supported in $D$, then we obtain from Proposition \ref{uniformboundT}, 
\begin{equation}\label{3.30}
\norm{w_{\ell_2,m_2}}_{L^2(B)}\leq C\norm{(I-T_{A,q})^{-1}v_{\ell_2,m_2}}_{H^1(D)}\leq C\norm{v_{\ell_2,m_2}}_{H^1(D)}.
\end{equation}
We note that there is a constant $C>0$ such that the inequality 
\begin{equation}\label{3.31}
\norm{u}_{H^1(D)}\leq C\norm{u}_{L^2(B)}
\end{equation}
holds true for all $u\in H^1_{loc}(\R^3)$ satisfying the Helmholtz equation $\Delta u+k^2u=0$ in $\R^3$. We can then estimate
\begin{equation}\label{3.32}
\norm{w_{\ell_2,m_2}}_{L^2(B)}\leq C\norm{v_{\ell_2,m_2}}^2_{H^1(D)}\leq \norm{v_{\ell_2,m_2}}^2_{L^2(B)}\leq C\Big(\frac{eka }{2\ell_2+1}\Big)^{2\ell_2+3},
\end{equation}
we conclude, from (\ref{3.25}), (\ref{3.27}) and (\ref{3.32}), that
\begin{align}\label{3.33}
\abs{\mu_{(\ell_1, m_1; \ell_2, m_2)}}^2 & \leq C \norm{v_{\ell_1,m_1}}_{L^2(B)}^2 \norm{w_{\ell_2,m_2}}^2_{L^2(B)}\cr
& \leq C \Big(\frac{ eka }{2\ell_2+1}\Big)^{2\ell_2+3} \Big(\frac{eka }{2\ell_1+1}\Big)^{2\ell_1+3}.
\end{align}
This competes the proof.
\end{proof}

The following lemma makes the link between  $u^s_{A,q}(\cdot, y)$ and $u^\infty_{A,q}(\cdot, d)$. The proof follows similar ideas as in Stefanov \cite{texbook26.} for $A=0$ but uses different arguments since we do not rely on the properties of the Green function for $A \neq 0$ when $x\sim y$. Our proof would apply to more general contexts since we mainly rely on the reciprocity relation in Lemma \ref{L3.1}. The proof is given in Appendix B.
\begin{lemma}\label{L4.4.} 
The scattered field associated with point sources can be expanded as 
\begin{equation}\label{4.33.}
u^s_{A,q}(x,y)=-\frac{k^2}{4\pi} \underset{\underset{(\ell_2,m_2)\in \Gamma}{(\ell_1,m_1)\in \Gamma}}{\sum} i^{\ell_1-\ell_2}\mu_{(\ell_1, m_1; \ell_2, m_2)} h^{(1)}_{\ell_1}(k\vert x\vert) h^{(1)}_{\ell_2}(k\vert y\vert) Y^{m_1}_{\ell_1}\left( \hat x\right)Y^{m_2}_{\ell_2}\left( \hat y\right),
\end{equation}
uniformly for $|x|, |y| \ge a$, with $ \hat x = x/|x|$ and $\hat y = y/|y|$.
\end{lemma}
Then we have the following Lemma showing the Lipschitz continuity of the mapping  $ u_{A,q}^{\infty}  \mapsto  \mathcal{N}_{A,q}$ when $u_{A,q}^{\infty} $ is endowed with the norm \eqref{1.19}.
\begin{lemma}\label{L3.6}  Let $A_j$ and $q_j$, $j=1,2$ be as in Section 2.  Then 
\begin{equation}\label{3.37}
\Vert \mathcal{N}_{A_1,q_1}-\mathcal{N}_{A_2,q_2}\Vert \leq \alpha^2 \frac{k^2}{4\pi} \Vert u_{A_1,q_1}^\infty-u_{A_2,q_2}^\infty \Vert_{\mathcal{F}},
\end{equation}
\end{lemma}
\begin{proof} For $j=1,2$, denote by  $\mu^{j}_{(\ell_1, m_1; \ell_2, m_2)}$  the Fourier coefficients associated with  $u_{A_j,q_j}^\infty$ as above. 
 We get from \eqref{4.33.}
\begin{multline*}\label{3.44}
\Vert \mathcal{N}_{A_1,q_1} -\mathcal{N}_{A_2,q_2} \Vert^2\cr
 \le \left( \frac{k^2}{4\pi}\right)^2\!\!\sum_{\substack{(\ell_1,m_1)\in\Gamma \\ (\ell_2,m_2)\in\Gamma}}  | h_{\ell_1}^{(1)}(ka)|^2 |h_{\ell_2}^{(1)}(ka)|^2  |\mu^1_{(\ell_1, m_1; \ell_2, m_2)}-\mu^2_{(\ell_1, m_1; \ell_2, m_2)}|^2.
\end{multline*}
The estimate then follows using \eqref{3.20}
\end{proof}
From (\ref{3.37}) and Theorem \ref{T1.1} we easily derive Theorem \ref{T1.2}. Following the same arguments as in \cite{texbook3} one can obtain a stability result using only the $L^2$ norm of the far field. In fact, identity \eqref{4.33.} and the uniform bound of Corollary \ref{C2.3} allows us to reproduce exactly the same arguments as in \cite[Section 4]{texbook3} to state the following continuity result. 
\begin{lemma} \label{tioe}Let $M>0$ and $0<\theta<1$ be given.
Let $A_j\in W^{1,\infty}(\R^3,\R^3)$ and $q_j\in L^\infty(\R^3)$ be as in Section 2 such that $\|A_j\|_{W^{1,\infty}} \le M$ and $\|q_j\|_{L^{\infty}} \le M$. Then there exists a constant $\rho>0$  that only depends on $M$, $k$, $a$ and $\theta$ and a constant $\omega$ that only depends on $a$ and $k$ such that
\begin{equation*}
\Vert \mathcal{N}_{A_1,q_1} -\mathcal{N}_{A_2,q_2}\Vert \leq \rho^2 \exp \Big( -\big( - \ln \frac{\Vert u_{A_1,q_1}^{\infty}-u_{A_2,q_2}^{\infty}\Vert_{L^2(\s^2 \times \s^2)}}{\omega \rho}\big)^\theta\Big),
\end{equation*}
where $ \mathcal{N}_{A_j,q_j} $, $j=1,2$ denote here the near field operators associated with $B=\{x \in \Rt, |x|<2 a\}$.
\end{lemma}
Using the result of this lemma and  Theorem \ref{T1.1} one  can prove Theorem \ref{T1.3}   as follows. 
\subsubsection*{Proof of Theorem \ref{T1.3}}
According to Lemma \ref{tioe}
\begin{align*}
-\textrm{ln}\left(\Vert \mathcal{N}_{A_1,q_1} -\mathcal{N}_{A_2,q_2}\Vert \right) &\geq -\textrm{ln}(\rho^2)+\Big(-\textrm{ln}\frac{\Vert u^\infty_{A_1,q_1}-u^\infty_{A_2,q_2}\Vert_{L^2}}{\omega \rho}\Big)^\theta\cr
& \geq \frac{1}{2} \Big(-\textrm{ln}\frac{\Vert u^\infty_{A_1,q_1}-u^\infty_{A_2,q_2}\Vert_{L^2}}{\omega \rho}\Big)^\theta,
\end{align*}
for sufficiently small $\Vert u^\infty_{A_1,q_1}-u^\infty_{A_2,q_2}\Vert_{L^2}$ such that $ \Vert u^\infty_{A_1,q_1}-u^\infty_{A_2,q_2}\Vert_{L^2}\leq \omega\rho e^{-2(\textrm{ln}(\rho^2))^{\frac{1}{\theta}}}$.
Then, if we further suppose that $ \Vert u^\infty_{A_1,q_1}-u^\infty_{A_2,q_2}\Vert_{L^2}\leq e/(\omega\rho)$, then
\begin{align*}
\left( -\textrm{ln}\left(\Vert \mathcal{N}_{A_1,q_1} -\mathcal{N}_{A_2,q_2}\Vert \right)\right)^{-\frac{\sigma}{\sigma +3}}\leq 2^{\frac{\sigma (1-\theta)}{\sigma +3}} \left( -\textrm{ln}\left( {\Vert u^\infty_{A_1,q_1}-u^\infty_{A_2,q_2}\Vert_{L^2}}\right)\right)^{-\frac{\sigma \theta}{\sigma +3}}.
\end{align*}
Using the first inequality in Theorem \ref{T1.1} and choosing $\theta$ such that $\theta\frac{\sigma}{\sigma +3}=\frac{\sigma}{\sigma +3}-\epsilon$, where $0<\epsilon< \frac{\sigma}{\sigma +3}$, yield the first inequality of Theorem \ref{T1.3} related to $\Vert\mathrm{curl}A_1 - \mathrm{curl}A_2\Vert_{L^\infty (D)}$. The estimate for $\Vert q_1-q_2\Vert_{L^\infty (D)}$ is derived analogously.


\appendix

\section{Proof of Lemmas \ref{L4.2} and \ref{L4.4}}
\setcounter{equation}{0}
This Appendix is devoted to the proof of the Lemma \ref{L4.2}. At first, we recall the following three lemmas proved in \cite{texbook6} on the properties of the operator $N_\omega^{-1}$, $\omega\in  \sss +i\sss$, given by (\ref{4.2}). The first Lemma, due to Salo \cite{texbook25} in the reconstruction methods and similar to the one appearing in Eskin and Ralston \cite{texbook1} and Sun \cite{texbook9}, shows that a relation between a non-linear and linear Fourier transform of $\omega\cdot A$ for a vector field $A$.
\begin{lemma}\label{LA.1}
Let $\omega=\omega_1+i\omega_2$ with $\omega_1, \omega_2\in \sss$ and $\xi\in \Rt$, such that $\xi$, $\omega_1$ and $\omega_2$ be three mutually orthogonal vectors in $\R^3$. Let $A\in W^{2,\infty}(D)^3$ with $\mathrm{Supp}(A)\subset D$. Then we have the following equality
\begin{equation}\label{A.1}
\int_{\R^3}\omega\cdot A(x)e^{ix\cdot\xi}e^{iN^{-1}_\omega (-\omega \cdot A)(x)}dx =\int_{\R^3} \omega \cdot A(x) e^{ix\cdot\xi }dx,
\end{equation}
where $A$ is extended by $0$ outside $D$.
\end{lemma}
\begin{lemma}\label{LA.2}
Let $g\in W^{n,\infty}(\Rt)$, $n\geq 0$, with $\textrm{Supp}(g)\subseteq D$. Then $N^{-1}_\omega (g)\in W^{n,\infty}(\R^3)$ and satisfies
$$ \norm{N_\omega^{-1}(g)}_{W^{n,\infty}(\R^3)} \leq C\norm{g}_{W^{n,\infty}(D)},$$
where $C$ depends only on $D$.
\end{lemma}
Finally, we have the following result which gives the dependence of $N^{-1}_\omega(-\omega \cdot A)$ on the parameter $\omega$.
\begin{lemma}\label{LA.3}
Let $A\in W^{2,\infty}(D)$ with  $\mathrm{Supp}(A)\subset D$ such that $\Vert A\Vert_{W^{2,\infty}} \leq M$ and let $\theta,\,\theta' \in \sss +i\sss$ such that $\vert \theta-\theta'\vert <1$ such that $\Re(\theta) \cdot \Im(\theta)=\Re(\theta') \cdot \Im(\theta')=0$. Then, we have the following inequality
\begin{equation}\label{A.2}
\Vert N^{-1}_{\theta}(-\theta \cdot A)-N^{-1}_{\theta'}(-\theta'\cdot A)\Vert_{L^\infty (D)}\leq C \abs{\theta- \theta'},
\end{equation}
where $C$ depends only in $D$ and $M$.
\end{lemma}
\begin{proof}[Proof of Lemma \ref{L4.2}]
By using (\ref{4.8}), we have for $j=1,2$
\begin{equation}\label{A.3}
\nabla  u_j=e^{ix\cdot\rr_j}(i(\rr_j+\nabla \varphi_j)e^{i\varphi_j}+i\rr_jr_j+\nabla r_j),
\end{equation}
where $\varphi_1$ and $\varphi_2$ are given by
\begin{equation}\label{A.4}
\varphi_1(x,\omega_1^*)=N_{\omega_1^*}^{-1}(\omega_1^*\cdot A_1),\quad \varphi_2(x,\omega_1^*)=N_{\omega_2^*}^{-1}(-\omega_2^*\cdot A_2),
\end{equation}
where $A_j$, $j=1,2$ are extended by $0$ outside $D$.\\
Therefore, direct calculation gives
\begin{equation}\label{A.5}
u_2 \nabla u_1 -u_1\nabla u_2 = i(\rr_1-\rr_2)e^{i(\varphi_1+\varphi_2)}e^{ix\cdot(\rr_1+\rr_2)}+\Psi_1 (x, \rr_1, \rr_2)+\Psi_2(x,\rr_1,\rr_2),
\end{equation}
where $\Psi_1$ and $\Psi_2$ are given by
\begin{align*}
\Psi_1(x,\rr_1,\rr_2)&=i(\rr_1-\rr_2)\left[r_1e^{i\varphi_2}+r_2e^{i\varphi_1}+r_1r_2\right]e^{ix\cdot(\rr_1+\rr_2)},\\
\Psi_2(x,\rr_1,\rr_2)&=\big[i(\nabla \varphi_1-\nabla \varphi_2)e^{i(\varphi_1+\varphi_2)}+(\nabla r_1e^{i\varphi_2}
-\nabla r_2 e^{i\varphi_1})\\
&\qquad\qquad+(ir_2\nabla \varphi_1e^{i\varphi_1}-ir_1\nabla \varphi_2e^{i\varphi_2})+(r\nabla r_1-r_1\nabla r_2)\big]e^{ix\cdot(\rr_1+\rr_2)}.
\end{align*}
Using Lemma \ref{LA.2} and the fact that $A_j$, $j=1,2$, is compactly supported in $\R^3$, we obtain
\begin{equation}\label{A.6}
\norm{\varphi_j(\cdot,\omega_j^*)}_{L^\infty(D)}\leq C\Vert A_j\Vert_{L^\infty(\R^3)}\leq CM. 
\end{equation} 
This implies that 
\begin{equation}\label{A.7}
\norm{\Psi_1(\cdot,\rr_1,\rr_2)}_{L^1(D)}\leq C.
\end{equation}
From (\ref{A.5}), we obtain
\begin{equation}\label{A.8}
i\int_{\Rt} A\cdot \left(  u_2\nabla u_1-u_1 \nabla u_2 \right) dx=\int_{\Rt} A\cdot (\rr_2-\rr_1)e^{ix\cdot\xi}e^{i(\varphi_1+\varphi_2)} dx+\mathcal{R}_1(\xi ,s),
\end{equation}
where 
\begin{equation}\label{A.9}
\mathcal{R}_1(\xi ,s)=i\int_{\Rt} A\cdot (\Psi_1(x,\rr_1,\rr_2)+\Psi_2(x,\rr_1,\rr_2))dx,\quad \xi =\rr_1+\rr_2.
\end{equation}
Let now compute the first integral in the right hand side of (\ref{A.8}). By using (\ref{4.6}) and (\ref{4.9}), we have for $\omega=\omega_1 +i\omega_2$
\begin{multline}\label{A.10}
\int_{\Rt} A\cdot (\rr_2-\rr_1)e^{ix\cdot\xi}e^{i(\varphi_1+ \varphi_2)} dx =2s\int_{\Rt}\overline{\omega}\cdot  Ae^{ix\cdot\xi}e^{i(\varphi_1+\varphi_2)}dx\cr
+2s\Big( \sqrt{1-\frac{\vert \xi \vert^2}{4s^2}}-1\Big) \int_{\Rt} \omega_1\cdot Ae^{ix\cdot\xi}e^{i(\varphi_1+\varphi_2)}dx,
\end{multline}
Let $\psi_1=N_{\overline{\omega}}^{-1}(\overline{\omega} \cdot A_1)$, and $\psi_2=N^{-1}_{\overline{\omega}} (-\overline{\omega}\cdot A_2)$, then we have 
$$
\psi_1(x)+\psi_2(x)=N^{-1}_{\overline{\omega}}(-\overline{\omega}\cdot (A_2-A_1))=N^{-1}_{\overline{\omega}}(-\overline{\omega}\cdot A).$$
We insert $e^{i(\psi_1+\psi_2)}$ in (\ref{A.10}), then we have
\begin{equation}\label{A.11}
\int_{\Rt} A(x)\cdot (\rr_2-\rr_1)e^{ix\cdot\xi}e^{i(\varphi_1+ \varphi_2)} dx 
=\mathcal{J}(\xi, s)+\mathcal{R}_2(\xi, s)+\mathcal{R}_3(\xi, s),
\end{equation}
where
\begin{align*}
\mathcal{J}(\xi, s)&=2s\int_{\Rt}\overline{\omega}\cdot A(x)e^{ix\cdot\xi}e^{i(\psi_1+\psi_2)}dx, \\
\mathcal{R}_2(\xi, s)&=2s \int_{\Rt}\overline{\omega}\cdot A(x)e^{ix\cdot\xi}\left(e^{i(\varphi_1+\varphi_2)}- e^{i(\psi_1+\psi_2)}\right) dx, \\
\mathcal{R}_3(\xi, s)&=2s\Big( \sqrt{1-\frac{\vert \xi \vert^2}{4s^2}}-1\Big) \int_{\Rt} \omega_1\cdot A(x)e^{ix\cdot\xi}e^{i(\varphi_1+\varphi_2)}dx.
\end{align*}
By using the Lemma \ref{LA.1}, we obtain 
\begin{equation}\label{A.12}
\mathcal{J}(\xi, s)=2s\int_{\R^3} \overline{\omega} \cdot A(x) e^{ix\cdot \xi}dx.
\end{equation}
On the other hand, we have
\begin{align}\label{A.13}
\mathcal{R}_2(\xi,s)=-2s\int_\Rt e^{-ix\cdot\xi}\overline{\omega}\cdot A\Big(e^{i\varphi_1}\big(e^{i\varphi_2}-e^{i\psi_2}\big)-e^{i\psi_2}\big(e^{i\psi_1}-e^{i\varphi_1}\big)\Big)dx.
\end{align}
Using the dependence of $N^{-1}_{\omega}(-\omega\cdot A)$ on the parameter $\omega$ given in Lemma \ref{LA.3} and the fact that $\textrm{Supp}A\subset D$, we get
\begin{align}\label{A.14}
\abs{e^{i\varphi_2}-e^{i\psi_2}}&\leq C\abs{N_{\omega_2^*}^{-1}(-\omega_2^*\cdot A_2)-N_{\overline{\omega}}^{-1}(-\overline{\omega}\cdot A_2)}\leq C\abs{\omega_2^*-\overline{\omega}},\cr
\abs{e^{i\psi_1}-e^{i\varphi_1}}&\leq C\abs{N_{\overline{\omega}}^{-1}(\overline{\omega}\cdot A_1)-N_{-\omega_1^*}^{-1}(-\omega_1^*\cdot A_1)}\leq C\abs{\overline{\omega}+\omega_1^*}.
\end{align}
Taking into account (\ref{A.13}), (\ref{A.14}) and using that $1-\sqrt{1-\vert\xi\vert^2/s^2}\leq\vert\xi\vert^2/4s^2$, for all $\vert\xi\vert\leq 2s$, we conclude
\begin{equation}\label{A.16}
\big\vert\mathcal{R}_2(\xi,s)\big\vert\leq Cs\frac{\vert\xi\vert^2}{4s^2}\leq C\vert\xi\vert.
\end{equation}
By the same way, we find $\vert\mathcal{R}_3(\xi,s)\vert\leq C\vert\xi\vert$, for some positive constant which is independent of $\xi$ and $s$.\\
The proof is completed.
\end{proof}
\begin{proof}[Proof of Lemma \ref{L4.4}]
By a direct calculation, we have
\begin{equation}\label{A.17}
u_1u_2=
e^{ix\cdot\xi}e^{i(\varphi_1+\varphi_2)}+e^{ix\cdot\xi}\big(e^{i\varphi_{2}}r_{1}
+e^{i\varphi_{1}}r_{2}+r_1r_2\big).
\end{equation}
We use the identity (\ref{A.17}) and we insert $e^{ix\cdot\xi} q(x)$, we obtain
\begin{equation}\label{A.18}
\int_{D}e^{-i\varphi}q(x)u_{1}u_2 dx= \int_{D}
q(x)\;e^{ix\cdot\xi}\,dx+\mathcal{R}'_{1}(\xi,s)+\mathcal{R}'_{2}(\xi,s),
\end{equation}
where
\begin{align*}
\mathcal{R}'_{1}(\xi,s)&=\int_{D}q(x)\;e^{ix\cdot\xi}\,e^{i\varphi_{1}}
\para{e^{i(\varphi_{2}-\varphi)}-e^{-i\varphi_{1}}}dx,\cr
\mathcal{R}'_{2}(\xi,s)&=\int_{D}e^{-i\varphi}q(x)e^{ix\cdot\xi}\big(e^{i\varphi_{2}}r_{1}
+e^{i\varphi_{1}}r_{2}+r_1r_2\big)dx.
\end{align*}
Let $\psi_{3}=N^{-1}_{\omega_{2}^{*}}(-\omega_{2}^{*}\cdot A_{1}).$ We insert $e^{i\psi_{3}}$ in  $\mathcal{R}'_{1}$ and obtain from Lemmas \ref{LA.2} and \ref{LA.3}
\begin{align}\label{A.19}
\abs{\mathcal{R}'_{1}(\xi,s)}&\leq C\norm{e^{i(\varphi_{2}-\varphi)}-e^{i\psi_{3}}}_{L^{\infty}(D)}+\norm{e^{i\psi_{3}}-e^{-i\varphi_{1}}}_{L^{\infty}(B)}\cr
&\leq  C\big(\norm{N^{-1}_{\omega_{2}^{*}}(-\omega_{2}^{*}\cdot
(A_{2}+\nabla\varphi))-N^{-1}_{\omega_{2}^{*}}(-\omega_{2}^{*}\cdot
A_{1})}_{L^{\infty}(B)}\cr
&\qquad \qquad +\norm{N^{-1}_{\omega_{2}^{*}}(-\omega_{2}^{*}\cdot
A_{1})-N^{-1}_{-\omega_{1}^{*}}(\omega_{1}^{*}\cdot
A_{1})}_{L^{\infty}(B)}\big)\cr 
&\leq C\para{\norm{A+\nabla\varphi}_{L^{\infty}(B)}+\abs{\omega_{2}^*+\omega_{1}^*}}\cr
&\leq  C\para{\norm{A+\nabla\vartheta}_{L^{\infty}(B)}+\norm{\vartheta}_{W^{1,\infty}(B\backslash B')}+\abs{\omega_{2}^*+\omega_{1}^*}}.
\end{align}
Using (\ref{AD12}) and (\ref{AD13}), we obtain 
\begin{equation}\label{A.20}
\abs{\mathcal{R}'_{1}(\xi,s)} \leq  C\para{\norm{\mathrm{curl}(A)}_{L^{\infty}(B)}+s\seq{\xi}^{-1}}.
\end{equation}
Moreover from (\ref{4.9}), we get $\abs{\mathcal{R}'_{2}(\xi,s)}\leq C s^{-1}$. Collecting this with (\ref{A.20}) and (\ref{A.18}) we obtain the desired result.
\end{proof}
\section{Proof of Lemma \ref{L4.4.}}

This appendix is devoted to the proof of Lemma \ref{L4.4.}.  We first establish the following duality result. The coefficients $A$ and $q$ are as in Section \ref{sec2}.
\begin{lemma}\label{L.I}
Let $v_i\in H^1(B)$, $i=1,2$ such that $\Delta v_i +k^2v_i=0$ in $B$. Then, the following identity 
\begin{equation}
\int_B Q_{A,q}v_1 (I -T_{-A,q})^{-1}v_2dx= \int_B Q_{-A,q}v_2 (I -T_{A,q})^{-1}v_1dx,
\end{equation}
holds where $Q_{A,q}$ and $T_{A,q}$ are given by \eqref{1.2} and \eqref{1.9} respectively.
\end{lemma}
\begin{proof}
We associate to $v_1$ (respectively $v_2$) a total field  $u_{A,q}$ (respectively $u_{-A,q}$) and a scattered field $u^s_{-A,q}$ (respectively $u^s_{-A,q}$). Applying (2.33) to $u_1=u^s_{A,q}$ and $u_2=u^s_{-A,q}$ implies
$$\int_{B}\para{\mathcal{H}_{A,q} u^s_{A,q} \, u^s_{-A,q}-u^s_{A,q}\mathcal{H}_{-A,q} u^s_{-A,q}}dx=0.$$
Making use of $u_{A,q}=u^s_{A,q}+v_1$  and $u_{-A,q}=u^s_{-A,q}+v_2$, it follows by direct calculation
\begin{align*}
\int_B \left(Q_{A,q}v_1u_{-A,q}-Q_{-A,q}v_2\right) dx=\int_B \left( Q_{A,q}v_1v_2-Q_{-A,q}v_2v_1 \right) dx.
\end{align*}
Moreover, by integrating by parts, we get 
$$\int_B\left( Q_{A,q}v_1v_2-Q_{-A,q}v_2v_1 \right)dx=0. $$
Due to $u_{A,q}=(I-T_{A,q})^{-1}v_1$ and $u_{-A,q}=(I-T_{-A,q})^{-1}v_2$, we obtain the desired result.
\end{proof}
\begin{proposition}\label{P1}
For $k>0$ fixed, we have 
\begin{align}
u^s_{A,q}(x,y)&=\frac{1}{4\pi} \frac{e^{ik\vert x\vert}}{\vert x\vert} \frac{e^{ik\vert y\vert}}{\vert y\vert} u^\infty_{A,q}\left( \widehat{x},-\widehat{y} \right)+ \frac{1}{\vert x\vert \vert y\vert}\left(\frac{1}{\vert x\vert}+\frac{1}{\vert y\vert}\right) \Lambda(x,y), \hspace*{0.5cm}x\neq y,
\end{align}
where $\Lambda(x,y)$ is uniformly bounded as $\vert x\vert \longrightarrow \infty$ and $\vert y\vert \rightarrow \infty$.
\end{proposition}
\begin{proof}
Using the asymptotics of $\Phi(\cdot ,z)$ and $\Dd_z \Phi (\cdot,z)$ for $z\in D$, we get
\begin{align}
\label{1.5a}
u_{A,q}^s(x,y)&=\int_D Q_{-A,q}\Phi (x,z)u_{A,q}(z,y)dz\\\label{1.5b}
&=\frac{e^{ik\vert x\vert}}{4\pi \vert x\vert} w\left( y, -\widehat{x}\right) +\frac{O(1)}{|x|^2} \|u_{A,q}(.,y)\|_{L^2(D)}, \hspace*{0.5cm} \vert x\vert \rightarrow \infty ,
\end{align}
uniformly with respect to $y \in \R^3\setminus B$, where
\begin{align*}
w(y,d):=\int_D Q_{-A,q}e^{ikz\cdot d} u_{A,q}(z,y)dz.
\end{align*}
Using now Lemma \ref{L.I} and integrating by parts, we get 
\begin{equation}
w(y,d)=- \int_D \Phi (z,y)Q_{-A,q}u_{-A,q}(z,d )dz,
\end{equation}
and therefore $w(\cdot,d)=  u_{-A,q}^s(\cdot, d)$. Consequently, \eqref{1.5b} and  \eqref{3.18}  imply that for $\vert x\vert$, $ \vert y\vert \rightarrow \infty$
 \begin{align}\label{1.9.}
u_{A,q}^s(x,y) &=\frac{1}{4\pi} \frac{e^{ik\vert x\vert}}{\vert x\vert} \frac{e^{ik\vert y\vert}}{\vert y\vert } u^\infty_{-A,q}\left(\widehat{y} , -\widehat{x} \right)+ \frac{O(1)}{|x|^2} \|u_{A,q}(\cdot,y)\|_{L^2(D)}\cr
&\qquad +  \frac{O(1)}{|x||y|^2} \|u_{-A,q}(\cdot,-\hat x)\|_{L^2(D)}.
\end{align}
Let us observe that according to Corollary \ref{C2.3} (and the asymptotic behaviour of $\Phi(x,y)$  and $\Dd_x \Phi (x,y)$ with $x\in D$ and $\vert y\vert \rightarrow \infty$) we get that 
$$
\|u_{A,q}(\cdot,y)\|_{L^2(D)}= \frac{O(1)}{|y|}\quad \mbox{ as }  \vert y\vert \rightarrow \infty,
$$ 
and  $\|u_{-A,q}(\cdot,-\hat x)\|_{L^2(D)} $ is uniformly bounded with respect to $\hat{x}$.  We finally obtain the desired result by noticing the reciprocity relation  $u^\infty_{A,q}(d, \theta)=u^\infty_{-A,q}(-\theta, -d)$ where $\theta ,d \in \mathbb{S}^2$ (which is also a consequence of Lemma \ref{L.I}) or using Lemma \ref{L3.1}.
\end{proof} 

Now, let us expand the scattering amplitude $u^\infty_{A,q}(d , \theta )$ in spherical harmonics
\begin{equation}
u^\infty_{A,q}(d ,\theta )= \underset{(\ell_1,m_1)\in \Gamma}{\sum}\; \underset{(\ell_2,m_2)\in \Gamma}{\sum}  \mu_{\ell_1 m_1\ell_2m_2}Y^{m_1}_{\ell_1}(d) Y^{m_2}_{\ell_1}(\theta ),
\end{equation}
where $\mu_{\ell_1 m_1\ell_2m_2}$ is given by \eqref{3.17}. 
\subsubsection*{Proof of Lemma \ref{L4.4.}}
Making use of the addition formula \cite{texbook2}, 
\begin{align}
\Phi (x,z)&=\underset{\ell,m}{\sum }\varepsilon_{\ell,m}(z) h^{(1)}_\ell (k\vert x\vert ) Y^m_\ell \left(\widehat{x}\right), \hspace*{0.5cm} \vert x\vert > \vert z\vert,\\
\nabla_y \Phi (x,z)&=\underset{\ell,m}{\sum} \varepsilon^\prime_{\ell,m} (z)h^{(1)}_{\ell}(k\vert x\vert) Y^m_\ell\left( \widehat{x}\right),\hspace*{0.5cm} \vert x\vert > \vert z\vert,
\end{align}
where $\varepsilon_{\ell,m}(z)= ikj_\ell(k\vert z\vert) Y^m_\ell \left( \widehat{z}\right)$ and $\varepsilon^\prime_{\ell,m}=ik \nabla \left( j_\ell(k\vert z\vert )Y^m_\ell\left( \widehat{z}\right)\right)$, it follows from \eqref{1.5a} that for  $y\in \R^3 \setminus D$ and uniformly for $ \vert x\vert \ge a$
\begin{align*}
u^s_{A,q}(x,y)&= \underset{(\ell_1,m_1)\in \Gamma}\sum \alpha_{\ell_1m_1}(y) h^{(1)}_{\ell_1}(k\vert x\vert)Y^{m_1}_{\ell_1}\left(\widehat{x}\right).
\end{align*}
Similarly, for $x\in \R^3 \setminus D$ and uniformly for $ \vert y\vert \ge a$
\begin{align*}
u^s_{-A,q}(y,x)&= \underset{(\ell_2,m_2)\in \Gamma}\sum \beta_{\ell_2m_2}(x) h^{(1)}_{\ell_2}(k\vert y\vert)Y^{m_2}_{\ell_2}\left(\widehat{y}\right).
\end{align*}
We observe that 
$$
\alpha_{\ell_1m_1}(y)  h^{(1)}_{\ell_1}(ka) = \int_{{\mathbb S}^2} u^s_{A,q}(a \hat{x},y)  \overline{Y^{m_1}_{\ell_1}}\left(\widehat{x}\right) ds(\hat{x}).
$$
Using the reciprocity relation of Lemma \ref{L3.1} we then get uniformly for $ \vert y\vert \ge a$
$$
\alpha_{\ell_1m_1}(y)   = \underset{(\ell_2,m_2)\in \Gamma}\sum \gamma_{\ell_1m_1 \ell_2m_2} h^{(1)}_{\ell_2}(k\vert y\vert)Y^{m_2}_{\ell_2}\left(\widehat{y}\right),
$$
where 
$$
\gamma_{\ell_1m_1 \ell_2m_2} = \frac{1}{h^{(1)}_{\ell_1}(ka)} \int_{{\mathbb S}^2} \beta_{\ell_2m_2}(a \hat{x})  \overline{Y^{m_1}_{\ell_1}}\left(\widehat{x}\right) ds(\hat{x}).
$$
This yields in particular that for $ \vert x\vert \ge a$ and $ \vert y\vert \ge a$
\begin{equation}
\label{expansionUS}
u^s_{A,q}(x,y)= \sum_{\substack{(\ell_1,m_1)\in \Gamma \\ (\ell_2,m_2)\in \Gamma}} \gamma_{\ell_1m_1 \ell_2m_2}  h^{(1)}_{\ell_1}(k\vert x\vert)Y^{m_1}_{\ell_1}\left(\widehat{x}\right) h^{(1)}_{\ell_2}(k\vert y\vert)Y^{m_2}_{\ell_2}\left(\widehat{y}\right).
\end{equation}
Observe that from Proposition \ref{P1}, we have in particular that $u^s_{A,q} \in L^2(S_R \times S_R)$ for  $R$ sufficiently large, with $S_R = \{|x|=R\}$. The orthonormality of products of spherical harmonics in $L^2({\mathbb S}^2 \times {\mathbb S}^2)$ implies in particular that the series \eqref{expansionUS} is convergent in $L^2({\mathbb S}^2 \times {\mathbb S}^2)$ and 
$$
 \int_{{\mathbb S}^2\times {\mathbb S}^2 } u^s_{A,q}(R \hat{x}, R \hat{y}) \overline{Y^{m_1}_{\ell_1}}\left(\widehat{x}\right)  \overline{Y^{m_2}_{\ell_2}}\left(\widehat{y}\right) ds(\hat x) ds(\hat{y}) =   \gamma_{\ell_1m_1 \ell_2m_2} h^{(1)}_{\ell_1}(kR)  h^{(1)}_{\ell_2}(kR).
$$
We recall that (\cite{texbook2}),
$$
h^{(1)}_\ell(r)=(-i)^{\ell+1}\frac{e^{ir}}{r}+O\left( r^{-2}\right),
$$
Proposition \ref{P1} and the identity $Y^{m_2}_{\ell_2}(-\hat y )=(-1)^{\ell_2} Y^{m_2}_{\ell_2}(\hat y )$ then imply, by integrating on $S_R \times S_R$ and letting $R\to \infty$ that 
$$
\gamma_{\ell_1m_1 \ell_2m_2} = \frac{-k^2}{4\pi} i^{\ell_1+\ell_2} (-1)^{\ell_2}  \mu_{\ell_1m_1 \ell_2m_2}.
$$
This yields the desired result. We remark that it also implies from \eqref{3.23} and \eqref{3.20} that the series \eqref{expansionUS} is absolutely convergent together with its first and second derivatives with respect to $x$ or $y$.


\begin{thebibliography}{99}




\bibitem{texbook4} H. Ben Joud, \emph{A stability estimate for an inverse problem for the Schr\"odinger equation in a magnetic field from partial boundary measurements}, Inverse Problems 25, 045012 (23 pp), (2009).

\bibitem{Caro-Pohjola} P. Caro and V. Pohjola, \emph{Stability estimates for an inverse problem for the magnetic Schr\"odinger operator}, International Mathematics Research Notices, 2015, (21), 11083-11116, (2015).

\bibitem{texbook2} D. Colton, and R. Kress, \emph{Inverse Acoustic and Electromagnetic Scattering Theory}, 3rd ed. New York, Springer, (2013).

\bibitem{DKSU2007}  D. Dos Santos Ferreira, C. E. Kenig, J. Sj\"ostrand and G. Uhlmann, \emph{Determining a magnetic Schr\"odinger operator from partial Cauchy data}, Comm. Mathematical Physics, 271, 467-488 (2007).

\bibitem{texbook1} G. Eskin and J. Ralston, \emph{Inverse scattering problem for the Schr\"odinger equation with magnetic potential at a fixed energy}, Commun. Math. Phys. 173, 199-224, (1995).

\bibitem{Grisv}P. Grisvard, Pierre, \emph{Elliptic Problems in Nonsmooth Domains}, Society for Industrial and Applied Mathematics, (2011).

\bibitem{texbook10} P. H\"ahner, \emph{Scattering by media}. Scattering, Academic Press, pp. 74-94,  (2002).

\bibitem{texbook3} P. H\"ahner, and T. Hohage, \emph{ New stability estimates for the inverse acoustic inhomogeneous medium problem and applications}. SIAM journal on mathematical analysis, 33(3), 670-685, (2001).

\bibitem{texbook28} M.I. Isaev, R.G. Novikov, \emph{New global stability estimates for monochromatic inverse acoustic scattering}, 
 SIAM Journal on Mathematical Analysis, 45 (3), pp. 1495-1504, (2013).



\bibitem{texbook22b} K. Krupchyk \emph{Inverse Transmission Problems for Magnetic Schr\"odinger Operators}, {International Mathematics Research Notices}, 2014(1), pp. {65-164}, (2012).

\bibitem{texbook22} K. Krupchyk and G. Uhlmann, \emph{Uniqueness in an inverse boundary problem for a magnetic Schr\"odinger operator with a bounded magnetic potential}, Comm. Math. Phys., 327, pp. 993-1009, (2014).

\bibitem{lionsmagenes} J.L. Lions and E. Magenes, Non-homogeneous Boundary Value Problems and Applications, vols. I, II, Springer-Verlag, Berlin, 1972.



\bibitem{texbook5.} C. J. Mero{\~n}o, L. Potenciano-Machado, and M. Salo, \emph{The fixed angle scattering problem with a first order perturbation,} arXiv preprint arXiv:2009.13315, (2020).

\bibitem{texbook017} A. Nachman, \emph{Reconstructions from boundary measurements}, Ann. of Math. 128 (1988) 531-576.


\bibitem{NSU95} G. Nakamura, Z. Sun, G. Uhlmann, \emph{Global identifiability for an inverse problem for
the Schr\"{o}dinger equation in a magnetic field}, Math. Ann., 303, 377-388 (1995).

\bibitem{texbook023} R.G. Novikov, \emph{A multidimensional inverse spectral problem for the equation $-\Delta\psi+(v(x)-Eu(x))\psi=0$}, Funct. Anal. Appl. 22 (1988) 263-272.





\bibitem{texbook026} A.G. Ramm, \emph{Recovery of the potential from fixed energy scattering data}, Inverse Problems 4 (1988) 877-886.

\bibitem{P-MR2018} L. Potenciano-Machado and A. Ruiz \emph{Stability estimates for a magnetic Schrödinger operator with partial data}, {Inverse Problems \& Imaging}, 12(6), 1309-1342 (2018).

\bibitem{texbook25} M. Salo, \emph{Semiclassical pseudo-differential calculus and the reconstruction of a magnetic field}, Commun, PDE 31, 1639-66, (2006).

\bibitem{texbook5} V. Serov, and J. Sandhu, \emph{Scattering solutions and Born approximation for the magnetic Schr\"odinger operator}, Inverse Problems in Science and Engineering, 27, Issue 4, 422 - 438, (2019).

\bibitem{texbook26.} P. Stefanov, \emph{Stability of the inverse problem in potential scattering at fixed energy}, Ann. Inst. Fourier (Grenoble) 40 (1990), 867-884.

\bibitem{texbook9} Z. Sun, \emph{An inverse boundary value problem for the Schr\"odinger operator with vector potentials}, Trans. Amer. Math. Soc., 338, 953-969 (1992).

\bibitem{texbook030} J. Sylvester, G. Uhlmann, \emph{A global uniqueness theorem for an inverse boundary value problem}, Ann. of Math. 125 (1987),
153-169.

\bibitem{texbook6}  L. Tzou, \emph{Stability estimates for coefficients of magnetic Schr\"odinger equation from full and partial boundary measurements}, Communication in Partial Differential Equations 33, 1911-1952, (2008).
\end{thebibliography}
\end{document}